\documentclass[a4paper,11pt]{amsart}
\usepackage{fullpage}
\usepackage{lscape}
\usepackage{amsthm,float}
\usepackage{latexsym,bezier,caption,amsbsy,psfrag,color,enumerate,amsfonts,amsmath,amscd,amssymb,comment}
\usepackage{epsfig}
\usepackage{rotating}
\usepackage{graphics}
\usepackage{graphicx}
\usepackage{stmaryrd}
\usepackage{extarrows} 
\usepackage{amsmath}
\usepackage[foot]{amsaddr}
\usepackage{cite}
\usepackage{dsfont}
\usepackage{mathtools}
\usepackage{turnstile}
\usepackage{tikz}
\usepackage{url}
\usepackage{hyperref}
\DeclareEmphSequence{\bfseries,\mdseries}
 
\DeclareSymbolFont{rsfscript}{OMS}{rsfs}{m}{n}
\DeclareSymbolFontAlphabet{\mathrsfs}{rsfscript}

\DeclareMathOperator{\Aut}{Aut}

\DeclareMathOperator{\Sch}{Sch}

\DeclareMathOperator{\Cay}{Cay}
\DeclareMathOperator{\Star}{Star}
\DeclareSymbolFont{rsfscript}{OMS}{rsfs}{m}{n}
\newtheorem{theorem}{Theorem}[section]
\newtheorem{prop}[theorem]{Proposition}
\newtheorem{defn}[theorem]{Definition}
\newtheorem{lemma}[theorem]{Lemma}

\newtheorem{cor}[theorem]{Corollary}

\newtheorem{quest}[theorem]{Question}
\newtheorem{rem}[theorem]{Remark}
\newtheorem*{theorem*}{Theorem}
\newtheorem{example}[theorem]{Example}
\newtheorem*{lehnert}{Lehnert Conjecture}
\newtheorem*{brough}{Brough Conjectures}

\newcommand{\la}{\langle}
\newcommand{\ra}{\rangle}

\def\vlongrightarrow{\relbar\joinrel\longrightarrow}
\def\vvlongrightarrow{\relbar\joinrel\vlongrightarrow}
\def\vvvlongrightarrow{\relbar\joinrel\vvlongrightarrow}
\def\vvvvlongrightarrow{\relbar\joinrel\vvvlongrightarrow}

\newcommand{\vvlongmapright}[1]{\smash{\mathop{\vvvlongrightarrow}\limits^{\raisebox{-0.7ex}{\tiny $#1$}}}}

\newcommand{\longfr}[2]{\smash{\stackrel{\text{\tiny{$#1|#2$}}}{\vlongrightarrow}}}
\newcommand{\vlongfr}[2]{\smash{\stackrel{\text{\tiny{$#1|#2$}\,}}{\vvlongrightarrow}}}

\newcommand{\vvvlongfr}[2]{\smash{\stackrel{\text{\tiny{$#1|#2$}\,}}{\vvvvlongrightarrow}}}

\def\vlongrightarrow{\relbar\joinrel\longrightarrow}

\newcommand{\mapright}[1]{%
  \smash{\stackrel{\raisebox{-0.9ex}{\tiny $#1$}}{\vlongrightarrow}}%
}

\title{Context-free graphs and their transition groups}



\author{Daniele D'Angeli$^{\ast}$} 
\address{$\ast$ Universit\`a Niccolo Cusano, Rome, Italy}
\email{daniele.dangeli@unicusano.it}
\author{Francesco Matucci$^{\dagger}$}
\address{$\dagger$ Universit\`{a} degli Studi di Milano--Bicocca, Milan, Italy}
\email{francesco.matucci@unimib.it}
\author{Davide Perego$^{\ddagger}$}
\address{$\ddagger$ IMUS, Universidad de Sevilla, Sevilla, Spain (Current: Section de mathématiques, Université de Genève, rue du Conseil-Général 7-9, 1205 Genève, Switzerland)}
\email{dperego9@gmail.com}
\author{Emanuele Rodaro$^{\S}$}
\address{$\S$ Politecnico di Milano, Milan, Italy}
\email{emanuele.rodaro@polimi.it}

\begin{document}
\begin{abstract}
Starting from context-free inverse graphs, we introduce a new class of groups and study their structural  properties. We establish closure properties, show that their co-word problems are context-free, analyze torsion elements, and realize them as subgroups of the asynchronous rational group. Context-freeness is preserved under a generalized free product of graphs, and using this construction we provide examples of groups that are not residually finite or not poly-context-free, making them relevant for testing the Lehnert and Brough conjectures. Moreover, we investigate how small local modifications of a graph affect the global structure of the transition group, showing that for locally quasi-transitive graphs with infinite orbits, the transition group decomposes into a highly structured quotient by a bounded torsion subgroup, showing strong global constraints induced by local graph properties.
\end{abstract}

\maketitle

\section{Introduction}
In combinatorial group theory the relationship among graphs, formal languages and groups and has been significantly advanced through the investigation of the word problem. Given a group $G$ generated by a finite symmetric set $A$, any element can be represented by a usually non-unique word over $A$. The \emph{word problem} of $G$ with respect to $A$ is the collection of words representing the identity, or, equivalently, the words that label every closed path in the Cayley graph of $G$ with respect to $A$. Throughout the vast literature, one of the main projects is to fully adapt the famous Chomsky hierarchy to finitely generated groups. The first efforts towards this goal has been made by Anisimov \cite{asimo}, who classified the class of all regular groups as that of all finite groups, and later by Muller and Schupp \cite{muahu,muahu2}, who have shown that the class of all context-free groups is that of all virtually free groups. Such results can be rephrased verbatim in the language of graph theory. Further developments of their renowned theorem were pursued by Ceccherini-Silberstein and Woess \cite{Tullio}, and later by Rodaro \cite{RodContextfree} to generalize it in a more combinatorial setting to a certain class of digraphs. We also recall that Herbst has proved that a group is one-counter if and only if it is virtually cyclic (again we have a generalization for this due to Holt, Owens and Thomas (\cite{HOT}) showing that finitely generated groups having an intersection of finitely many one-counter languages as word problem are virtually abelian) as a intermediate step in the characterization.

The next step of the hierarchy is that of context-sensitive groups. Currently, many intermediate classes have been provided with examples of groups belonging to them: poly-context-free (i.e., its word problem is a finite intersection of context-free languages), co-context-free, co-ET0L, and co-indexed, among others. Several famous examples, such as bounded automata groups and $\mathbb{Z}*\mathbb{Z}^{2}$, were recently exhibited (\cite{Bishop}, \cite{Raad}), in these particular cases they are co-ET0L. On the other hand, the classes of poly-context-free and co-context-free groups come with their universal container conjectures.

\begin{lehnert}\label{thm:lehnert}
A finitely generated group is co-context-free group if and only if it embeds into Thompson $V$.
\end{lehnert}

\noindent This conjecture is here restated in the formulation given by \cite{BleakMatu}.

\begin{brough}\label{cj: Brough}
The following conjectures are due to Brough:
\begin{enumerate}
    \item \label{cj: Brough1} A finitely generated group is poly-context-free if and only if it is a finitely generated subgroup of a virtual direct product of free groups.
    \item \label{cj: Brough2} A finitely generated solvable group is poly-context-free if and only if it is virtually abelian.
    \item \label{cj: Brough3} If a group $G$ is poly-context-free, then $G$ does not have arbitrarily large finite subgroups.
\end{enumerate}
\end{brough}
The main contributions of this paper lies in the fruitful interplay between group theory and graph theory. By leveraging this connection, we employ graph-theoretic techniques as a central tool in our construction. These techniques allow us to translate algebraic properties of groups into a combinatorial framework, which in turn enables a precise analysis of the associated formal languages. The deep interaction between graph structures, context-free languages and groups is well-established in the literature, as demonstrated by the aforementioned works in \cite{Tullio,RodContextfree, muahu,muahu2} and further evidenced by \cite{Lindorfer-Woess,Woess}.
In particular, starting from a subclass of tree-like graphs, namely, context-free inverse graphs $\Gamma$, which are, roughly speaking, deterministic involutive labeled digraphs with finitely many end-cones (see Section~\ref{sec:CF}), we define a new group, called the \emph{context-free transition group} $\mathcal{G}(\Gamma)$. This group is generated by the permutations of the vertices induced by the labeled directed edges of $\Gamma$. We then extend this construction to a finite collection of disjoint context-free graphs $\Gamma^{(1)},\ldots,\Gamma^{(k)}$, and denote the class of all such groups by {\bf CF-TR}. Although this class has a concrete combinatorial description, its algebraic properties remain largely unexplored, making it an intriguing subject of study in its own right. The paper is organized as follows.
Sections \ref{sec:preliminaires} and \ref{sec:CF}  are dedicated to give some background on inverse graphs and to recall the characterization of context-free inverse graphs. In Section~\ref{sec:transition-groups} we introduce our class of groups and prove some closure properties which allow us to describe any finitely generated subgroup of a virtual direct product of free groups as a transition group. Moreover, we prove the following result that frames our research in the context of the Lehnert Conjecture. 

\medskip
\noindent \textbf{Theorem~\ref{theo: co-cf transition group of more graphs}.}
\textit{Let $\Gamma^{(1)},\ldots,\Gamma^{(k)}$ be context-free inverse $A$-digraphs that are also complete, then the word problem of $\mathcal{G}(\Gamma^{(1)} \sqcup \ldots \sqcup\Gamma^{(k)})$ is a co-context-free language.}
\medskip

It is well-known that Grigorchuk's group and $\mathbb{Z}*\mathbb{Z}^{2}$ (see \cite{BleakOlga}) are two possible counterexamples to \hyperref[thm:lehnert]{Lehnert Conjecture}. Even though the first is not a transition group by virtue of the following result proved in Section~\ref{sec:torsion}.

\medskip
\noindent \textbf{Corollary~\ref{theo: torsion}.}
\textit{No infinite context-free transition group contains an infinite finitely generated torsion subgroup.}
\medskip

In the same section, in particular in Proposition~\ref{prop: order properties}, we show some upper bound of the order of an element with respect to its length, as a consequence we reprove the solvability of the torsion problem, which is already known for the broader class of co-context-free groups (see \cite{holt}). In Section~\ref{sec:poly-cf-transition}, we introduce the subclass of poly-inverse-context-free groups and establish some of their structural properties. In particular, we show that they can be realized as transition groups of context-free graphs.

\medskip
\noindent \textbf{Theorem~\ref{theo: poly with inverse graphs implies cf-trans}.}
\textit{Let $G$ be a group whose word problem is the intersection of context-free languages accepted by rooted inverse graphs. Then $G$ is in {\bf CF-TR}.}
\medskip

By assuming the graphs to be highly symmetric, that is, quasi-transitive, we are able to characterize their transition groups:

\noindent \textbf{Theorem~\ref{theo: virtually subgroups}.}
\textit{A group $G$ is the transition group of a collection of quasi-transitive context-free inverse $A$-digraphs if and only if $G$ is virtually a finitely generated subgroup of the direct product of free groups.}
\medskip

This last theorems fit into the framework of Brough's Conjectures. 
In \cite{GNS}, the rational group $\mathcal{R}$ of homeomorphisms of the Cantor set, described by sequential transducers (see the beginning of Section~\ref{sec:rationality}), was introduced. Interesting examples of groups embedding into $\mathcal{R}$ include the Thompson group $V$ and any hyperbolic group (\cite{BelkBleakMatu}). In this paper, we prove the following embedding theorem.

\medskip
\noindent \textbf{Theorem~\ref{theo:rational}.}
\textit{Let $G$ be a group in {\bf CF-TR}. Then $G$ embeds into the rational group $\mathcal{R}$.}
\medskip

In Section~\ref{subsec: perturbing}, we study how small local modifications of a graph affect its transition group, addressing the natural question of how local changes on the graph can perturb the global algebraic structure. A key result concerns the transition group $G$ of a locally quasi-transitive context-free inverse graph $\Gamma$ with infinite orbits. Roughly speaking, Theorem~\ref{theo: structural locally quasi-trans} shows that $G$ decomposes into a well-understood part and a “small” torsion part: there is a short exact sequence
\[
\mathds{1} \longrightarrow \mathrsfs{B}_{n} \longrightarrow G \longrightarrow H \longrightarrow \mathds{1},
\]
where $H$ is virtually a finitely generated subgroup of a direct product of free groups, and $\mathrsfs{B}_{n}$ is a torsion normal subgroup with bounded-order elements. In other words, $G$ is an extension of a highly structured group by a torsion subgroup, showing how local perturbations imposes strong global constraints on the transition group.

Motivated by the fact that the direct product of two transition groups is again a transition group (see Section~\ref{sec:transition-groups}), one may naturally ask whether the same holds for the free product. Although this question remains open, and a positive answer could solve Lehnert’s Conjecture, in Section~\ref{sec:free-product} we introduce a notion of free product for context-free graphs and show that the resulting product graph remains context-free (see Theorem~\ref{theo: free product is context-free}). Building on this construction and partially leveraging the techniques from Section~\ref{subsec: perturbing}, Section~\ref{sec: non-residually} presents three notable examples of context-free transition groups.
\begin{enumerate}
    \item A non-residually finite example. 
    \item The group $\mathbb{Z}\wr\mathbb{Z}$ which is non-finitely presented, solvable and not poly-context-free. In particular, this shows that the class of poly-context-free and \textbf{CoCF} are different.
    \item A group with an infinite torsion subgroup.
\end{enumerate}

Moreover, since we provide a non-residually finite transition group (see Section~\ref{sec: non-residually}), it is worth asking whether the class of context-free transition groups fully contains poly-context-free groups or serves as a tool to construct non-residually finite poly-context-free groups, which would disprove Brough's Conjecture~\ref{cj: Brough1}.


\section{Involutive and inverse labeled digraphs}\label{sec:preliminaires}

Let $A = \{a_1, \dots, a_n\}$ be a finite \emph{alphabet}. We assume throughout that $A$ is \emph{involutively closed}, meaning that $a \in A$ implies $a^{-1} \in A$. The free monoid on $A$ is denoted $A^*$, with identity $1$, and $A^+$ denotes the free semigroup. A word $w \in A^*$ is said to be \emph{reduced} if it contains no factor of the form $uu^{-1}$ for some $u \in A$. Every word $w \in A^*$ has a unique reduced form $\overline{w}$, obtained iteratively by deleting all such factors. The set $R(A)$ of reduced words is a regular language. The free group $F_A$ can be identified with $R(A)$, where the product is given by $u \cdot v = \overline{uv}$. With a slight abuse of notation we will denote the identity of $F_A$ by $1$, while for a generic group $G$, we will use the notation $\mathds{1}_G$.

\begin{defn}[Inverse graph]
An \emph{$A$-digraph} is a tuple $\Gamma = (V, E, \iota, \tau, \lambda)$ consisting of a vertex set $V$, directed edges $E$, and maps $\iota, \tau: E \to V$ giving the initial and terminal vertices of each edge, together with a labeling map $\lambda: E \to A$. We write $v_1 \mapright{a} v_2$ to denote an edge $e \in E$ with $\iota(e) = v_1$, $\tau(e) = v_2$ and $\lambda(e)=a$.
An \emph{involutive $A$-digraph} is a tuple $\Gamma = (V, E, \iota, \tau, \lambda, \cdot^{-1})$ such that $(V, E, \iota, \tau, \lambda)$ is an $A$-digraph and $\cdot^{-1}: E \to E$ is an involution (i.e., $(e^{-1})^{-1} = e$ for all $e \in E$) satisfying:
\begin{enumerate}
    \item $\lambda(e^{-1}) = \lambda(e)^{-1}$ for all $e \in E$;
    \item $\iota(e^{-1}) = \tau(e)$ and $\tau(e^{-1}) = \iota(e)$;
    \item $e \neq e^{-1}$ for all $e \in E$.
\end{enumerate}
Moreover, we say that $\Gamma$ is:
\begin{itemize}
    \item \emph{deterministic} if for each $v \in V$ and $a \in A$, there is at most one outgoing edge $v \mapright{a} v'$;
    \item \emph{complete} if for every $v\in V$ and $a\in A$ there is at least one outgoing $a$-edge.
\end{itemize}
An involutive $A$-digraph is called \emph{inverse} if it is deterministic and connected. 
\end{defn}
 A subgraph $\Gamma' = (V',E',\iota,\tau,\lambda, \cdot^{-1})$ of $\Gamma$ satisfies $V'\subseteq V$ and $E'\subseteq E$, and $e^{-1}\in E$ for all $e\in E'$, and we write $\Gamma' \subseteq \Gamma$. From now on, we will refer to inverse $A$-digraphs simply as $A$-graphs, and whenever the alphabet $A$ does not need to be specified, we will simply write graphs.

\begin{defn}
A \emph{walk} in $\Gamma$ is a finite or infinite sequence of edges $p = e_1 e_2 \cdots e_k \cdots$ such that $\tau(e_i)=\iota(e_{i+1})$ for all $i$. For a finite walk $p=e_1\cdots e_k$, the \emph{label} is $\lambda(p)=\lambda(e_1)\cdots\lambda(e_k)\in A^+$. The walk is said to be \emph{reduced} if it contains no subwalk of the form $e e^{-1}$. We denote the initial and terminal vertices of $p$ by $\iota(p)$ and $\tau(p)$; when $\Gamma$ is deterministic and $x\in V$, a word $w\in A^*$ determines at most one walk $x\mapright{w}y$. In case the walk $p$ has the property $\iota(p)=\tau(p)$, $p$ is called a \emph{circuit}. 
\end{defn}
When we fix a vertex $x_0$ in $\Gamma$, the pair $(\Gamma, x_0)$ is called a \emph{rooted graph}. In this case, we see $(\Gamma, x_0)$ as a \emph{language acceptor}
by using the distinguished vertex $x_0$ as an initial and final state. In this case, the language accepted by $(\Gamma,x_0)$ is given by the set $L(\Gamma,x_0)$ formed by the labels of all circuits with initial vertex $x_0$. Note that $L(\Gamma,x_0)$ together with the usual operation of concatenation, is also a submonoid of the free monoid $A^*$. More generally, an \emph{$A$-automaton} has an initial state and a set $F\subseteq V(\Gamma)$ of final states, the language $L(\Gamma, x_0,F)$ is the set formed by the labels of the walks $p$ with $\iota(p)=x_0$, $\tau(p)\in F$. In case $\Gamma$ is finite, $L(\Gamma, x_0,F)$ belongs to the class of regular languages \cite{hop}.\\
The class of $A$-graphs forms a category where a morphism 
\(\varphi: \Gamma^{(1)} \to \Gamma^{(2)}\) between the graphs
\(\Gamma^{(1)}=(V_1,E_1,\iota_1, \tau_1, \lambda_1, \cdot^{-1})\) to 
\(\Gamma^{(2)}=(V_2,E_2,\iota_2, \tau_2, \lambda_2,\cdot^{-1})\) is a pair of maps 
\(\varphi_e: E_1 \to E_2\), \(\varphi_v: V_1 \to V_2\) such that:  
\begin{itemize}
    \item \(\varphi\) respects the involution: \(\varphi_e(g^{-1}) = \varphi_e(g)^{-1}\) for all \(g \in E_1\);
    \item \(\varphi\) commutes with the incidence and labeling maps:
    \[
    \tau_2(\varphi_e(g)) = \varphi_v(\tau_1(g)), \quad
    \iota_2(\varphi_e(g)) = \varphi_v(\iota_1(g)), \quad
    \lambda_2(\varphi_e(g)) = \lambda_1(g).
    \]
\end{itemize}
Given an (inverse) subgraph \(\Lambda \subseteq \Gamma^{(1)}\), its image \(\varphi(\Lambda)\) is the (inverse) subgraph of \(\Gamma^{(2)}\) with vertices \(\varphi_v(V(\Lambda))\) and edges \(\varphi_e(E(\Lambda))\). An isomorphism (resp.\ monomorphism) is a morphism \(\varphi\) such that both \(\varphi_v\) and \(\varphi_e\) are bijective (resp.\ injective).
\medskip

\noindent\textbf{Metric structure.}
If $\Gamma$ is connected, we endow $V(\Gamma)$ with the usual path metric
\[
d(u,v)=\min\{\,\ell(p): p \text{ is a walk from } u \text{ to } v\,\},
\]
where $\ell(p)$ denotes the length (number of edges) of $p$. A walk realizing the minimum is called a \emph{geodesic}. For $v\in V(\Gamma)$ and $n\ge0$ we denote by
\[
D_n(v)=\{\,u\in V(\Gamma): d(v,u)\le n\,\}
\qquad\text{and}\qquad
S_n(v)=\{\,u\in V(\Gamma): d(v,u)=n\,\}
\]
the \emph{disk} (ball) and the \emph{sphere} of radius $n$, respectively, centered at $v$. For a connected subgraph $Y\subseteq\Gamma$ the diameter is
\[
\delta_\Gamma(Y)=\max\{d(y_1,y_2):y_1,y_2\in V(Y)\}.
\]
In interesting feature of inverse graphs is given by the following.
\begin{lemma}\label{lem: there are reduced walks}
Let $\Gamma$ be an $A$-graph. For any vertex $v\in V(\Gamma)$ and word $u\in A^*$ there is at most one walk $v\mapright{u}y$ labeled by $u$. Moreover, if such a walk exists then there is also a (reduced) walk $v\mapright{\overline{u}}y$ labeled by the free reduction $\overline{u}$ of $u$.
\end{lemma}

\subsection{Basic facts about morphisms of graphs}
The condition of determinism and connectedness of an inverse graph has some consequences when we consider morphisms in the category of $A$-graphs. For instance, a morphism $\psi:\Gamma\to\Lambda$ is completely determined by the vertex map $\psi_v$. In particular, $\psi$ is a monomorphism if for any $y\in V(\Lambda)$, the preimage $\psi_v^{-1}(y)$ contains at most one element and checking if two morphisms are equal is just a matter of verifying if they agree on a single vertex.
\begin{lemma}\label{lem: uniqueness of automorphism}
Let $\varphi_1, \varphi_2$ be two morphisms from the graph $\Gamma^{(1)}$ to the graph $\Gamma^{(2)}$, then $\varphi_1=\varphi_2$ if and only if $\varphi_1(x)=\varphi_2(x)$ for some $x\in V(\Gamma)$.
\end{lemma}
\begin{proof}
For any walk $x\mapright{u}y$, the images under $\varphi_1$ and $\varphi_2$ are $\varphi_1(x)\mapright{u}\varphi_1(y), \varphi_2(x)\mapright{u}\varphi_2(y)$. 
Determinism of $\Gamma^{(2)}$ and $\varphi_1(x)=\varphi_2(x)$ imply $\varphi_1(y)=\varphi_2(y)$. Since $\Gamma^{(1)}$ is connected, $\varphi_1$ and $\varphi_2$ coincide on all vertices, and hence also on all edges.
\end{proof}

Another interesting feature of graphs, and particularly rooted graphs, is the relationship between accepted languages and morphisms. This connection is contained in the following proposition belonging to folklore, see for instance  \cite[Proposition 1]{RodContextfree}. 
\begin{prop}\label{prop: inclusion implies homomorphism}
Let $(\Gamma,x_0), (\Lambda, y_0)$ be two rooted graphs, then $L(\Gamma,x_0)\subseteq L (\Lambda, y_0)$ if and only if there is a morphism $\psi: \Gamma\rightarrow  \Lambda$ with $\psi(x_0)=y_0$. Moreover, $L(\Gamma,x_0)= L (\Lambda, y_0)$ if and only if there is an isomorphism $\psi: \Gamma\rightarrow  \Lambda$. 
\end{prop}
Note that by the previous proposition, we have that, up to isomorphism, there is a unique rooted graph accepting the language $L(\Gamma,x_0)$. 
The set of automorphisms of the graph $\Gamma$ will be denoted by $\Aut(\Gamma)$. This is a group acting on $\Gamma$ in a canonical way. In case the quotient graph $\Gamma \setminus \Aut(\Gamma)$ is finite, we say that $\Gamma$ is \emph{quasi-transitive}. Observe that this last condition is equivalent to the finiteness of $V(\Gamma)\setminus \Aut(\Gamma)$. 
Any morphism $\varphi:\Gamma\to\Lambda$ between two graphs induces a bijections $\varphi:\Star(\Gamma,v)\to \Star(\Gamma',\varphi(v)) $ between the star sets in the obvious ways. Thus, morphisms of graphs are \emph{immersions} of graphs as defined by Stallings \cite{Stall}. In particular, if the graphs are complete, by Lemma~\ref{lem: there are reduced walks} every morphism between two complete $A$-graphs is a \emph{covering}, that is, an immersion of graphs $f: \Gamma \to \Lambda$ so that $f$ maps the sets $V(\Gamma)$ and $E(\Gamma)$ respectively onto $V(\Lambda)$ and $E(\Lambda)$. One central feature of coverings is the possibility of lifting walks, this is contained in the following characterization. 
\begin{prop}[Proposition 2.3 of \cite{Gr-Meak}] \label{prop: lifting of coverings}
Let $\varphi:\Gamma'\to\Gamma$ be a covering of graphs. Then, for each vertex $v\in V(\Gamma)$ and vertex $v'\in \varphi^{-1}(v)$, every walk $p$ in $\Gamma$ with $\iota(p)=v$ lifts to a unique walk $p'$ starting at $v'$, i.e., $\varphi(p')=p$. 
\end{prop}

\section{Context-free languages and graphs} \label{sec:CF}
There are several ways to define a context-free language, either via the notion of grammar, or using the notion of pushdown automaton. For the sake of conciseness, we briefly recall the main notions that are used in the paper, and we refer the reader to \cite{hop} for further details. 
\begin{defn}
A \emph{context-free grammar} is a tuple $(V,\Sigma,P,S)$ where $V$ is a finite set of non-terminal symbols (or variables), $\Sigma$ is the finite alphabet of terminal symbols that is disjoint from $V$, $S\in V$ is the starting symbol, and  $P\subseteq V\times (V\cup \Sigma)^*$ is a finite set called productions.
\end{defn}
Usually production are denoted by $X\Rightarrow \alpha$, $\alpha\in (V\cup \Sigma)^*, X\in V$.  We extend $\Rightarrow $ by putting $\alpha X \gamma\Rightarrow \alpha\beta\gamma$ whenever $X\Rightarrow \beta\in P$, and we denote by $\Rightarrow ^*$ the reflexive and transitive closure of $\Rightarrow $. The language generated by the grammar $(V,\Sigma,P,S)$ is $L(S)=\{u\in \Sigma^*: S\Rightarrow^*u\}$, and in this case we say that it is a \emph{context-free language}. A $L$ language with a context-free complement $\Sigma^* \setminus L$ is called \emph{co-context-free}, we will denote the class of context-free languages as \emph{CoCF}.\
 
In their paper, Muller and Schupp called a finitely generated group $G$ presented by
$\langle A \mid \mathcal{R} \rangle$ \emph{context-free} if the word problem $WP(G;A) = \{\, w \in A^* \mid \pi(w) = \mathds{1}_G\}$ is a context-free language, where $\pi : A^* \to G$ denotes the natural projection. With our notation, this language coincides with the set $L(\Cay(G;A),\mathds{1}_G)$ of words labeling a circuit starting and ending at the group identity. 
\begin{defn}
A rooted graph $(\Gamma, x_0)$ is called \emph{context-free} if $L(\Gamma, x_0)$ is a context-free language.
\end{defn}
This definition is equivalent to a more geometric one that involves the notion of end-cones. We briefly recall it. Henceforth, we denote by $d(x,y)$ the usual distance in the graph $\Gamma$, and put $\|v\|_{x_0}=d(x_0,v)$. In \cite{muahu2} Muller and Schupp introduced the central notion of end-cone together with the idea of end-isomorphism of such subgraphs. For a vertex $v$ with $\|v\|_{x_0}=n$, the \emph{end-cone at $v$}  is the subgraph $\Gamma(v,x_0)$ that is the connected component of $\Gamma\setminus D_{n-1}(x_0)$ containing $v$, and we denote by $\Delta(v,x_0)=\{y\in \Gamma(v,x_0): \|y\|_{x_0}=n\}$ the set of \emph{frontier vertices} of $\Gamma(v,x_0)$, that is, the vertices of $\Gamma(v,x_0)$ at distance $n$ from the root $x
_0$. An \emph{end-isomorphism} $\psi:\Gamma(v,x_0)\to \Gamma(w,x_0)$ between the end-cones $\Gamma(v,x_0), \Gamma(w,x_0)$ is an isomorphism of graphs with the property of preserving their frontier vertices: $\psi(\Delta(v,x_0))=\Delta(w,x_0)$. 

\begin{defn}
The rooted graph $(\Gamma, x_0)$ is called \emph{a context-free graph in the sense of Muller-Schupp} if up to end-isomorphism there are finitely many end-cones, i.e., there are graphs $\Gamma_1, \ldots, \Gamma_N$ such that for any end-cone $\Gamma(v,x_0)$ of $\Gamma$ there is an end-isomorphism $\psi:\Gamma(v,x_0)\to \Gamma_j$ for some $j\in [1,N]$.
\end{defn}
The previous two notions are equivalent: 
\begin{prop}[Proposition~5 of \cite{RodContextfree}] \label{prop: charact context-free}
For a graph $\Gamma$ the following conditions are equivalent:
\begin{enumerate}
\item for any non-empty finite set $F\subseteq V(\Gamma)$ the language $L(\Gamma, x_0,F)$ is deterministic context-free;
\item the language $L(\Gamma,x_0)$ is deterministic context-free, hence $(\Gamma, x_0)$ is context-free;
\item the rooted graph $(\Gamma, x_0)$ is context-free in the sense of Muller-Schupp;
\item for any vertex $y_0\in V(\Gamma)$, the graph $(\Gamma, y_0)$ is context-free in the sense of Muller-Schupp.
\end{enumerate}
\end{prop}
The languages accepted by context-free graphs can be characterized more precisely through the construction of suitable pushdown automata, called inverse PDA. In \cite{RodContextfree}, such languages are called \emph{inverse-context-free}, and we will use this term throughout the paper.

\section{Transition groups of graphs and co-context-free groups} \label{sec:transition-groups}

In this section, we introduce a new subclass of co-context-free groups.
For an $A$-graph $\Gamma$, the set of \emph{everywhere-circuits} is 
\begin{align}\label{eq:everywhere-circuits}
\mathcal{L}(\Gamma)=\{u\in A^*: \forall p\in V(\Gamma)\,\mbox{ there is a circuit } p\mapright{u}p\mbox{ in } \Gamma\}=\bigcap_{p\in V(\Gamma)} L(\Gamma, p)
\end{align}
It is easy to see that if $\Gamma$ is an $A$-graph that is also complete, then $\mathcal{L}(\Gamma)$ is a normal subgroup of $F_A$. The following is a key notion in our paper.
\begin{defn}
 For a complete $A$-graph $\Gamma$, each letter $a\in A$ defines a permutation on the set of vertices $a: V(\Gamma)\to V(\Gamma)$. The \emph{transition group} $\mathcal{G}(\Gamma)$ of the graph $\Gamma$ is the group generated by composing all these permutations. 
\end{defn}
\begin{rem}
    Note that this definition is a specialization of the well-known notion of transition monoid for an automaton to the case of inverse automata. This is also isomorphic to the syntactic monoid of the language $L=L(\Gamma, x_0)$ (where $x_0$ is any chosen vertex of $\Gamma$). Indeed, it is not difficult to see that if one considers the syntactic congruence $\sim_{\Gamma}$ on $A^*$ defined by $u\sim_{L} v$ if for all $x,y\in A^*$ $xuy\in L$ if and only if $xvy\in L$, then since $\Gamma$ is inverse, we have that $A^*/_{\sim_{L}}$ is isomorphic to $\mathcal{G}(\Gamma)$. If we remove the completeness condition, the syntactic monoid is an inverse monoid, see \cite{eil76, Saka}. 
\end{rem}
The group $\mathcal{G}(\Gamma)$ may be equivalently obtained by considering the quotient of the free group $F_A$ generated by $A$ by the normal subgroup $\mathcal{L}(\Gamma)$. There is a natural right action of $\mathcal{G}(\Gamma)$ on the set of vertices $V(\Gamma)$: for any $g\in \mathcal{G}(\Gamma)$ and vertex $v\in V(\Gamma)$, $v\cdot g\in V(\Gamma)$ is the vertex for which there is a walk $v\mapright{w} (v\cdot g)$ in $\Gamma$ for any $w\in A^*$  that represents $g$. We sometimes also put $v\cdot w=v\cdot g$ for any $w\in A^*$ representing $g$ in $\mathcal{G}(\Gamma)$.

\begin{example}\label{thm:example-Z-finite}
For any group $G$ on the symmetric generating set $A$, it is straightforward to check that the transition group of the Cayley graph $\Cay(G;A)$ is isomorphic to $G$. In particular, if $\Gamma$ is either a finite $n$-cycle or 
an infinite line and by suitably labeling all edges either by $a$ or $a^{-1}$ from a single-letter alphabet $A=\{a\}$,
we can turn $\Gamma$ into a complete graph so that its
transition group is readily seen to be either isomorphic to $\mathbb{Z}/n\mathbb{Z}$ or $\mathbb{Z}$, so that all these groups are transition groups. 
\end{example}

More in general, the rooted graph $(\Gamma, p)$ may be naturally identified with the Schreier graph of the stabilizer subgroup $H=\mathcal{G}(\Gamma)_p=\{g\in \mathcal{G}(\Gamma): g(p)=p\}$. This fact follows from Proposition~\ref{prop: inclusion implies homomorphism} and the equality $L(\Gamma, p)=L(\Sch(H,A), H)$. On the other hand, given a group $G$ generated by $A$ and a subgroup $H$, the Schreier graph of $H$, i.e. the graph which vertices are right coset of $H$ and direct edges are labeled by elements in $A$, is an example of a complete $A$-graph. In view of this, one can provide an algebraic characterization of transition groups of a graph.

\begin{prop} \label{prop:schreier}
A group $G$ is a transition group for a complete $A$-graph $\Gamma$ if and only $G$ admits a core-free subgroup $H$.  
\end{prop}
\begin{proof}
Let $p$ be any vertex of $\Gamma$ and set $H := G_p$, the stabilizer of $p$. Since for any $g \in G$ the conjugate  $G_p^g = G_{p \cdot g}$, the normal core of $H$ (the intersection of all conjugates of $H$) coincides with the intersection of all vertex stabilizers in $\Gamma$. By the definition of a transition group, this intersection is trivial. Conversely, consider the Schreier graph $\Sch(H,A)$ of $H$. As previously discussed, it is a complete $A$-digraph on which $G$ acts naturally. Let $w$ be a word in $\mathcal{L}(\Sch(H,A),H)$. By definition, $w$ represents an element of $G$ that stabilizes all vertices of the Schreier graph; that is, $Hg \cdot \overline{w} = Hg$ for all $g \in G$. Therefore, $\overline{w}$ belongs to the normal core of $H$, which is trivial, so $w \in WP(G,A)$. This shows that $\mathcal{L}(\Sch(H,A),H) \subseteq WP(G,A)$. The reverse inclusion is straightforward, hence $G$ is indeed the transition group of $\Sch(H,A)$.
\end{proof}

\begin{rem}
By the previous proposition we immediately deduce that the growth of the group $\mathcal{G}(\Gamma)$ is bounded below by the growth of the graph $\Gamma$.
\end{rem}

We actually aim to work in a slightly more general setting. Let $\{\Gamma^{(1)}, \ldots, \Gamma^{(k)}\}$ be a finite collection of complete $A$-graphs. As before, we can consider the permutations of vertices induced by the letters of $A$. This allows us to define the \emph{transition group} associated with the collection as $\mathcal{G}(\Gamma^{(1)} \sqcup \ldots \sqcup \Gamma^{(k)}) := F_A /_{\mathcal{L}(\Gamma^{(1)}, \ldots, \Gamma^{(k)})}$, where
\[
\mathcal{L}(\Gamma^{(1)}, \ldots, \Gamma^{(k)}) := \bigcap_{i=1}^k \mathcal{L}(\Gamma^{(i)})
\]
denotes the set of words acting trivially on all the graphs in the collection.
It is convenient to view a transition group over a finite union of graphs as a subgroup of the direct product of the corresponding transition groups over single graphs. More precisely, we have the following fact.

\begin{lemma} \label{lem:subgroup-product}
Let $\Gamma^{(1)}, \ldots, \Gamma^{(k)}$ be complete $A$-graphs. Then $\mathcal{G}(\Gamma^{(1)} \sqcup \ldots \sqcup \Gamma^{(k)}) \hookrightarrow \mathcal{G}(\Gamma^{(1)}) \times \ldots \times \mathcal{G}(\Gamma^{(k)})$, i.e., the transition group of the union is isomorphic to a subgroup of the direct product of the individual transition groups.
\end{lemma}
\begin{proof}
Let $G = \mathcal{G}(\Gamma^{(1)} \sqcup \ldots \sqcup \Gamma^{(k)})$. Define the map $\psi: G \to \mathcal{G}(\Gamma^{(1)}) \times \ldots \times \mathcal{G}(\Gamma^{(k)})$ by putting $g=[u] \mapsto (\pi_1(u), \ldots, \pi_k(u))$, where $\pi_i: A \to \mathcal{G}(\Gamma^{(i)})$ are the natural projections. By definition $WP(G;A) = \mathcal{L}(\Gamma^{(1)},\ldots,\Gamma^{(k)}) = \bigcap_{i=1}^k \mathcal{L}(\Gamma^{(i)})$, so $\psi$ is well-defined. Moreover, $\psi$ is injective: if $\pi_i(u) = \mathds{1}$ for all $i$, then $u \in \bigcap_i \mathcal{L}(\Gamma^{(i)}) = WP(G;A)$, hence $g = \mathds{1}$ in $G$.
\end{proof}

\subsection{Transition groups of context-free graphs}
We now show that the transition group of any context-free graph is {\bf CoCF}.
\\
Before proving the next lemma we need to recall the following definition. For a language $L\subseteq A^*$, the \emph{shift} of $L$, is the language $S_{A}(L)=\{xy\in A^*: yx\in L\}$ formed by all the words that may be obtained by a cyclic shift of some word of $L$. If it is clear from the context, we drop the subscript $A$ and we simply refer to $S(L)$. It is a well-known fact (see for instance Ex. 6.4 of \cite{hop}) that for a context-free language $L$, $S(L)$ is also context-free. For the next lemma, we need the following definition. Let $\pi = v \xrightarrow{u^{\omega}} \ldots$ be an infinite walk, for some $u \in A^*$. We say that a vertex $p$ is \emph{colored by a prefix $u'$ of $u$} if $u = u' u''$ and
\[
\pi = v \xrightarrow{u^s} v' \xrightarrow{u'} p \xrightarrow{u''} p' \xrightarrow{u^{\omega}} \ldots
\]
for some integer $s \ge 0$.
\begin{lemma} \label{lem:infinite-path}
Let $\mathcal{G}(\Gamma)$ be transition group of a complete context-free graph $\Gamma$ and let $g$ be a non trivial element of $\mathcal{G}(\Gamma)$. There exists an end-cone $\Gamma(v,x_0)$ s.t. the infinite walk $\pi=v \mapright{\mathring{g}^{\omega}}\ldots \subseteq \Gamma(v,x_0)$ with $\mathring{g} \in S(\{g\})$. Moreover, for any vertex $p\in \pi$ colored by a prefix $g'$ such that $g=g'g''$ we have $p\cdot (g''g')\neq p$.
\end{lemma}
\begin{proof}
Since $g$ is nontrivial, there exists a vertex $q_1 \in V(\Gamma)$ and a walk $q_1 \xrightarrow{g} q_2$ with $q_1 \neq q_2$. Consider the infinite one-sided walk $\pi' = q_1 \xrightarrow{g^{\omega}} \ldots$, which exists since $\Gamma$ is complete. 
Choose a vertex $v$ in $\pi'$ belonging to the frontier vertices $\Delta(v, x_0)$ such that $\pi' \subseteq \Gamma(v, x_0)$ (this is possible since every vertex belongs to some frontier set, and $\pi' \subseteq \Gamma(x_0, x_0) = \Gamma$). Equivalently, $v$ is chosen among the vertices of $\pi'$ that are at minimal distance from $x_0$.
Using the vertex $v$, we may find a cyclic shift $\mathring{g}$ of $g$ such that $\pi = v \xrightarrow{\mathring{g}^{\omega}} \ldots \subseteq \Gamma(v, x_0)$. 
Now, for any vertex $p \in \pi$ colored by a prefix $g'$ of $g$, with $g = g' g''$, it is straightforward to see that, by the determinism of $\Gamma$, the equality $p \cdot g'' g' = p$ would imply $q_1 = q_2$, a contradiction. Hence, necessarily $p \cdot g'' g' \neq p$.
\end{proof}

For the next theorem we need to know that the union of finitely many context-free languages is context-free, and for a deterministic context-free language the complement is also deterministic context-free, see \cite{hop}. 
\begin{theorem}\label{thm:inverse-digraph-is-CoCF}
Let $\Gamma$ be a context-free $A$-graph that is also complete, then $\mathcal{L}(\Gamma)\in {\bf CoCF}$.
\end{theorem}
\begin{proof}
Let $\{\Gamma_1, \ldots, \Gamma_N\}$ be the end-cones types of $\Gamma$ and let $\Delta_i$ be the corresponding frontier vertices. For each $q\in \Delta_i$ we denote by 
\[
N(q)=\{u\in A^*: q\mapright{u}p\mbox{ is a walk in } \Gamma_i\mbox{ with }p\neq q\}
\]
the set of words that do not label a circuit at $q$ in $\Gamma_i$. We claim:
\[
\mathcal{L}(\Gamma)^c=\bigcup_{i=1}^N \left ( \bigcup_{q\in \Delta_i} S\left(N(q)\right)\right )
\]

Let us prove that the right hand side is contained in the left hand side. Note that every walk $q\mapright{u}p$ in $\Gamma_i$ that does not label a circuit at $q$, has the same property in $\Gamma$ in any subgraph that is end-isomorphic to $\Gamma_i$, i.e., $u\in \mathcal{L}(\Gamma)^c$. Moreover, take the word $yx$ such that $xy=u\in N(q)$, and factor a walk $q\mapright{u}p$ with $p\neq q$ in $\Gamma$ as $q\mapright{x}v\mapright{y}p$. If $yx$ labeled a circuit at $v$, then by the determinism of $\Gamma$ we would have $p=q$, a contradiction. Hence, $yx\in S(N(q))$ does not label a circuit at $v$ in $\Gamma$, i.e., $yx\in \mathcal{L}(\Gamma)^c$.\\
Let us now prove the other inclusion. So let $g\in \mathcal{L}(\Gamma)^c$, by Lemma~\ref{lem:infinite-path} there exists an end-cone $\Gamma(v,x_0)$ s.t. the infinite walk $v \mapright{\mathring{g}^{\omega}}\ldots \subseteq \Gamma(v,x_0)$ with $\mathring{g} \in S(\{g\})$ and $v \cdot \mathring{g} \neq v$.
Hence, $g\in S(N(w))$ where $w$ is the frontier vertex in the end-cone type $\Gamma_j$ that is end-isomorphic to the end-cone $\Gamma(v,x_0)$, and this concludes the proof of the other inclusion. To show that $\bigcup_{i=1}^N \left ( \bigcup_{q\in \Delta_i} S\left(N(q)\right)\right )$ is context-free, it is enough to show that each $N(q)$ is context-free, since context-free languages are closed under the cyclic-shift operator, and a finite union of context-free languages is context-free. Now, for any $q\in\Delta_i$, $N(q)=L(\Gamma_i, q)^c$ and since $\Gamma_i$ is, in particular, a context-free graph, we have that $L(\Gamma_i, q)$ is a deterministic context-free language by Proposition~\ref{prop: charact context-free}. Now, since deterministic context-free languages are closed under complementation, we conclude that $L(\Gamma_i, q)^c$ is also context-free, and this concludes the proof of the proposition. 
\end{proof}

\begin{rem}
In the previous proof we have shown that each language $N(q)$ is a deterministic context-free language. If $S(N(q))$ were deterministic for all $q\in\Delta_i$, then, by taking the complement, we would deduce that $\mathcal{L}(\Gamma)$ is an intersection of finitely many context-free languages, i.e., $\mathcal{L}(\Gamma)$ is poly-context-free. 
However, it seems that the cyclic closure of a context-free language is generally far from being deterministic, making the class of transition groups of context-free graphs quite different from that of poly-context-free groups. In fact, we are going to find an example of a transition group which is not poly-context-free (see Definition~\ref{def:comb}).

\end{rem}
As an immediate consequence of the previous proposition, we obtain the following theorem. 
\begin{theorem}\label{theo: co-cf transition group}
Let $\Gamma$ be a context-free $A$-graph that is also complete. Then, the transition group $\mathcal{G}(\Gamma)$ is {\bf CoCF}. 
\end{theorem}
In the theorem above, the connectedness of $\Gamma$ is not strictly necessary and since {\bf CoCF} languages are closed under finite intersections we have 
\begin{theorem}\label{theo: co-cf transition group of more graphs}
If $\Gamma^{(1)}, \ldots, \Gamma^{(k)}$ are context-free $A$-graphs that are complete, then, the transition group $\mathcal{G}(\Gamma^{(1)} \sqcup\ldots \sqcup \Gamma^{(k)})$ is {\bf CoCF}.
\end{theorem}
\begin{defn}[{\bf CF-TR} groups]
In view of Theorem~\ref{theo: co-cf transition group of more graphs}, we may define the subclass of {\bf CoCF} groups formed by all the transition groups of a finite disjoint collection of context-free $A$-graphs, denoted by {\bf CF-TR}.
\end{defn}

\begin{quest}\label{prob: general problems cf-trans-group}
There are several natural questions for such a class of groups.
\begin{enumerate}
\item What is the relationship with the class of poly-context-free groups? 
\item Is the class {\bf CF-TR} strictly included into {\bf CoCF}?
\end{enumerate}
\end{quest}
There is a natural characterization of such groups in terms of certain coverings. We say that a covering $\psi:\Lambda\to \Gamma$ is \emph{transition-preserving} if for any for any walk $q'\mapright{u}p'$ with $q'\neq p'$ in $\Lambda$ there is a walk $q\mapright{u}p$ in $\Lambda$ such that the corresponding walk $\psi(q)\mapright{u}\psi(p)$ in $\Gamma$ is not a circuit. We have the following fact.
\begin{prop}
If $\psi:\Lambda\to \Gamma$ is a transition-preserving covering between two complete $A$-graphs, then $\mathcal{G}(\Gamma)\simeq \mathcal{G}(\Lambda)$. In particular, a group $G$ is {\bf CF-TR} if and only if there is a transition-preserving covering of the Cayley graph $\Cay(G;A)$ onto a context-free $A$-graph. 
\end{prop}
\begin{proof}
Consider the map $\phi: \mathcal{G}(\Lambda)\to \mathcal{G}(\Gamma)$ defined as follows: for any $g\in \mathcal{G}(\Lambda)$ and $u\in A^*$ which represents $g$, we define $\phi(g)=h$ where $h\in \mathcal{G}(\Gamma)$ is the element defined by $u$. It is a well-defined map since if $v$ represents $g$ too, then $uv^{-1}\in \mathcal{L}(\Lambda)$ labels a circuit everywhere in $\Lambda$, and so also in $\Gamma$ since $\psi$ is a covering. Hence, $v$ represents $h$ as well. It is routine to check that it is a morphism of groups which is also surjective. Let us prove that $\phi$ is injective. Let $u\in \mathcal{L}(\Gamma)$, by the path lifting property (Proposition~\ref{prop: lifting of coverings}), any circuit $v\mapright{u}v$ in $\Gamma$ lifts to a walk $v'\mapright{u}v''$ with $\psi(v')=\psi(v'')=v$. Now, if $u\notin \mathcal{L}(\Lambda)$, by the transition-preserving property we would have that $u\notin \mathcal{L}(\Gamma)$, a contradiction.\

Let us prove the characterization. As observed in Example~\ref{thm:example-Z-finite}, it is straightforward to check that the transition group of the Cayley graph $\Cay(G;A)$ is isomorphic to $G$. Thus, if there is a covering $\psi:\Cay(G;A)\to \Gamma$ onto a context-free graph $\Gamma$, then by the first part of this proposition $\mathcal{G}(\Gamma)\simeq G$. Conversely, if $G=\mathcal{G}(\Gamma)$ and we fix a root $x_0$ in $\Gamma$, then the morphism $\psi: \Cay(G;A)\to \Gamma$ defined on the vertices by $\psi(v)=y$ whenever $1\mapright{g} v$ is a walk in $\Cay(G;A)$ and $y=x_0\cdot g$ and extended to a morphism of $A$-graphs, is a covering that is also transition-preserving. 
\end{proof}

\subsection{An example}\label{sec: example}
As a concrete example of a transition group of a context-free graph we consider the $A$-graph $\Omega$ on the alphabet $A=\{a,a^{-1},b,b^{-1},c,c^{-1}\}$ (see Fig. \ref{fig: omega}). The  set of vertices $V(\Omega)=\{p_i, q_i, i\in \mathbb{Z}\}$ and the set $E(\Omega)$ of edges is given by:
\[
p_i\mapright{a}p_{i+1},\; p_{i+1}\mapright{a^{-1}}p_i,\,q_i\mapright{b}q_{i+1},\; q_{i+1}\mapright{b^{-1}}q_i, \quad \forall i\in\mathbb{Z}
\]
\[
p_i\mapright{b}p_{i},\; p_{i}\mapright{b^{-1}}p_i,\,q_i\mapright{a}q_{i},\; q_{i}\mapright{a^{-1}}q_i, p_i\mapright{x}q_i, q_i\mapright{x}p_i, x\in\{c,c^{-1}\} \quad \forall i\in\mathbb{Z}
\]
We sometimes identify the generic vertex $p_n$ with the pair $(n,0)$, $n\in \mathbb{Z}$, while $q_m$ with the pair $(m,1)$, $m\in\mathbb{Z}$.

\begin{figure}[ht]
\includegraphics[scale=1.5]{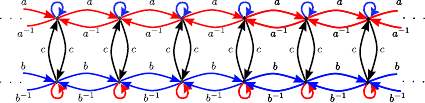}
\caption{The graph $\Omega$. The loops in red are labeled by $a,a^{-1}$, in blue by $b,b^{-1}$.} \label{fig: omega}
\end{figure}

Note that $\Omega$ is clearly quasi-transitive and it is quasi-isometric to a tree, thus by \cite[Theorem 2]{RodContextfree} it is a context-free graph. The next proposition characterizes the transition group defined by $\Omega$.
\begin{prop}\label{prop: the omega group}
The transition group of $\Omega$ has the following properties:
\begin{itemize}
\item the subgroup $\la a,b\ra$ generated by $a,b$ of the transition group $\mathcal{G}(\Omega)$ is isomorphic to $\mathbb{Z}^2$;
\item $\mathcal{G}(\Omega)$ is isomorphic to $G=\la a,b,c\mid c^2=1, ab=ba, ca=bc \ra $, so it is a semi-direct product $\mathbb{Z}^{2} \rtimes C_2$;
\item the action on $V(\Omega)$ is given by the formula $(x,y)\cdot a^nb^mc^j=(x+ym+(1-y)n,\, y+j\mod 2)$ where $a^nb^mc^j$ with $m,n\in\mathbb{Z}$, $j\in\{0,1\}$ is the standard normal form of the semi-direct product.
\end{itemize}
\end{prop}
\begin{proof}
Note that for any vertex $v$ in $\Omega$ we have $v\cdot ab=v\cdot ba$, hence $a,b$ commute in $\mathcal{G}(\Omega)$ ($ab=ba$ in $\mathcal{G}(\Omega)$). Similarly, one can check that for any vertex $v$ we have $v\cdot c^2=v$ and $v\cdot (ca)=v\cdot (bc)$. Hence, in $\mathcal{G}(\Omega)$ we also have the equalities $c^2=1, ca=bc$ and also $ac=cb$ as a consequence of the involution $c^2=1$. Using these equalities, it is easy to see that any element $g\in\mathcal{G}(\Omega)$ is equivalent to a word of the form $g=a^nb^m c^j$ with $m,n\in\mathbb{Z}$, $j\in\{0,1\}$. Now, every element $g\in G$ has a normal form $g=a^nb^m c^j$ with $m,n\in\mathbb{Z}$, $j\in\{0,1\}$. Thus, there is a well defined epimomorphism $\phi: G\to \mathcal{G}(\Omega)$ sending $a^nb^mc^j$ into the corresponding element in $\mathcal{G}(\Omega)$. The formula of the action can be obtained by observing that for a generic vertex $(x,y)\in V(\Omega)$ we have
\[
(x,y)\cdot a^n=(x+(1-y)n, y),\, (x,y)\cdot b^m=(x+ym, y),\, (x,y)\cdot c^j=(x,\,y+j\mod 2),\, 
\]
Let us show that $\phi$ is injective by using this action. Consider two elements $a^nb^m c^j, a^{n'}b^{m'}c^{j'}$, then for any vertex $(x,y)\in V(\Omega)$ the equation $(x,y)\cdot a^nb^mc^j= (x,y)\cdot a^{n'}b^{m'}c^{j'}$ implies $j=j'$ and if we choose a vertex with $y=0$ then we conclude that $n=n'$, while if we choose $y=1$ we get $m=m'$. 
\end{proof}

\subsection{Closure properties}
In this section we show that this class is closed under taking finitely generated subgroups, direct products, and finite extensions.
We fix $H$ a subgroup of $G$ that is finitely generated by $B \subseteq A^{*}$ and we set $n:= \max_{b\in B} \ell(b)$ with $\ell(b)$ the length in terms of $A$. 
\begin{lemma} \label{lem: infinite-path-B}
Let $g \in H$ be a nontrivial element. Then, there exist an end-cone $\Gamma(v,x_0)$ and a vertex $\widetilde{v} \in \Gamma(v,x_0)$ with $d(v,\widetilde{v}) \leq n$ such that the infinite walk $\widetilde{v} \xrightarrow{\mathring{g}^{\omega}} \ldots \subseteq \Gamma(v,x_0)$ for some $\mathring{g} \in S_{B}(\{g\})$ and $\widetilde{v} \cdot \mathring{g} \neq \widetilde{v}$.
\end{lemma}
\begin{proof}
By Lemma~\ref{lem:infinite-path}, there exists an end-cone $\Gamma(v, x_0)$ and a cyclic shift $\widehat{g} \in S(\{g\})$ such that the infinite walk $\pi = x \xrightarrow{\widehat{g}^{\omega}} \ldots$ is contained in $\Gamma(v,x_0)$. By the definition of $n$, possibly removing a prefix of $\pi$ of length at most $n$, we may choose a vertex $\widetilde{v} \in \pi$ such that $d(v,\widetilde{v}) \leq n$, colored by a prefix $g'$ of $g$, with $g = g'g''$, and set $\mathring{g} = g''g' \in S_{B}(\{g\})$. By construction, $\mathring{g}$ satisfies the required properties, namely $\widetilde{v} \xrightarrow{\mathring{g}^{\omega}} \ldots \subseteq \Gamma(v,x_0)$ and $\widetilde{v} \cdot \mathring{g} \neq \widetilde{v}$.
\end{proof}

\begin{prop}\label{prop: f.g. subgroups}
The class of {\bf CF-TR} is closed by taking finitely generated subgroups.
\end{prop}
\begin{proof}
Suppose first that we are dealing with one complete context-free graph, so that $H$ is a finitely generated subgroup of $\mathcal{G}(\Gamma)$, where $\Gamma$ is a graph that is also context-free. Let $B \subseteq A^{*}$ be the involutive closed generating set and let $\Delta_{1}, \ldots, \Delta_{N}$ be the frontier vertices of the end-cone types $\{\Gamma_1,\ldots,\Gamma_N\}$. We set $\Delta_{i}^{n}:=\{v \in \Gamma_i \mid d(x,v) \leq n \mbox{ for some } x \in \Delta_{i} \}$. With a slight abuse of notation we identify the frontier vertices of $\Delta_i$ with the frontier vertices of an end cone $\Gamma(v,x_0)$ of $\Gamma$ that is end-isomorphic to $\Gamma_i$. We do the same for the set $\Delta_{i}^{n}$. 
By definition of transition group, we need to show that $H$ is the transition group of a finite union of complete context-free $B$-graphs $\Theta_B$. We claim $$\Theta_B=\bigcup_{i=1}^{N} \bigcup_{v \in \Delta_{i}^{n}}(\Gamma^{B},v)$$ where $(\Gamma^{B},v)$ is given by the vertices $y$ of $\Gamma$ such that there exists $w \in B^{*}$ and $v\mapright{w}y$. The edges are $y_1\mapright{w}y_2$ whenever $w \in B$ and there exists a walk labeled by $w$ from $y_1$ to $y_2$ in $\Gamma$. It is not difficult to see that since $B$ is involutive closed and $\Gamma$ is inverse, then also $(\Gamma^{B},v)$ is inverse. Moreover, it is complete and context-free.\

Let us show that the complement of the word problem of $H$ and $\mathcal{G}(\Theta_B)$ coincide, as this will of course imply that their word problems coincide.
Let $ g \in G(\Theta_B)$ be a non trivial element, then there exists $\widetilde{v} \in (\Gamma^{B},v)$ for some $v \in \Delta_{i}^{n}$ such that $\widetilde{v}\cdot g \neq \widetilde{v}$. It follows that $\widetilde{v}\mapright{g}v'$ with $\widetilde{v} \neq v'$ is also a walk in $\Gamma$ and so $g\neq \mathds{1}_H$ in $H$. \

Conversely, suppose that $\mathds{1}_H \neq g \in H$, by Lemma~\ref{lem: infinite-path-B} there exists an end-cone $\Gamma(x,x_0)$ and a vertex $v \in \Gamma(x,x_0)$ with $d(x,v) \leq n$ such that $v\mapright{\mathring{g}^{\omega}} \ldots \subseteq \Gamma(x,x_0)$ and $v\cdot \mathring{g}\neq v$ for some $\mathring{g} \in S_{B}(\{g\})$. Now $\Gamma(x,x_0)$ is end-isomorphic to some $\Gamma_i$, hence we can find $\widetilde{v} \in \Delta_i^{n}$ such that $\widetilde{v}\mapright{\mathring{g}^{\omega}}\ldots  \subseteq \Gamma_i$ and $\widetilde{v}\cdot \mathring{g}\neq \widetilde{v}$. A fortiori, $\widetilde{v}\mapright{\mathring{g}^{\omega}}\ldots  \subseteq (\Gamma,\widetilde{v})$, whence $\widetilde{v}\mapright{\mathring{g}^{\omega}}\ldots \subset (\Gamma^{B}, \widetilde{v})$. Now in $(\Gamma^{B}, \widetilde{v})$, we can read $g$ on a walk $v' \mapright{u} \widetilde{v} \mapright{s} v''$ such that $\widetilde{v} \mapright{s} v''$ is a subwalk of $\widetilde{v}\mapright{\mathring{g}^{\omega}}\ldots$ and where $v' \neq v''$, otherwise determinism would imply $\widetilde{v}\cdot \mathring{g}=\widetilde{v}$. In particular, $g$ is non trivial in $G(\Theta_B)$.\

The previous argument may easily be adapted when considering a finite family of disjoint complete context-free graphs.
\end{proof}
Note that the previous proposition gives a constructive way to compute context-free graphs defining a finitely generated subgroup of a transition group of a context-free graphs $\Gamma$.
\begin{prop}\label{prop:closure_products}
The class of {\bf CF-TR} is closed by taking direct products.
\end{prop}
\begin{proof}
Let $\mathcal{G}(\Gamma_i)$ be a transition group of a context-free $A_i$-graph with $i=1,2$. We can assume that $A_1 \cap A_2$ is empty, otherwise we change accordingly the labels in $A_2$. Now we consider $\widehat{\Gamma}_1 \sqcup \widehat{\Gamma}_2$ where $\widehat{\Gamma}_i$ is a $A_1 \cup A_2$-graph in which for every vertex there are $|A_j|$ loops each labeled by an element in $A_j$ with $j \neq i$. It is straightforward that $\mathcal{G}(\widehat{\Gamma}_1 \sqcup \widehat{\Gamma}_2)$ is the desired direct product.
\end{proof}

The technique to prove the last closure property is inspired by \cite{holt} and, in particular, it is based on the \textit{inflated graph} construction introduced in \cite{RodContextfree} combined to the strategy in the proof of Proposition~\ref{prop: f.g. subgroups}.
\begin{defn}
Let $H$ be a finitely generated group which is a finite index subgroup of $G$ and let $T$ be a right transversal 
for $G$ with respect to $H$ containing the identity. For each $y \in A \cup T-\{\mathds{1}_{G}\}$ and $t \in T$, we fix a word $h_{t,y} \in A^{*}$ such that $ty=_{G} h_{t,y}\widetilde{t}$ 
for some $\widetilde{t} \in T$. We define the \emph{Schreier automaton} $\mathcal{A}$ over $A \cup T$ to be the finite state automaton with set of states $T$ , initial state $\mathds{1}_{G} \in T$ and $t\mapright{y \mid h_{t,y}}\widetilde{t}$ when $ty=_{G} h_{t,y}\widetilde{t}$.
\end{defn}
Starting from a Schreier automaton, we can build one that is inverse, i.e. 
\[
t \mapright{y \mid h_{t,y}} \widetilde{t} \ \Longleftrightarrow \widetilde{t} \ \mapright{y^{-1} \mid h^{-1}_{t,y}} t,
\]
without changing the groups involved $H$ and $G$. The way we do so is by enlarging the alphabet for the labels from $A \cup T$ to $A \cup \{T, T^{-1}\}$ (remember $A$ is already inverse by assumption) and by adding the inverse transitions $\widetilde{t} \ \mapright{y^{-1} \mid h^{-1}_{t,y}} t$. The fact that we are still dealing with $H$ and $G$ follows from the equation $\widetilde{t}y^{-1}=_{G} h^{-1}_{t,y}t$, which is true if and only if $ty=_{G} h_{t,y}\widetilde{t}$. We call it the \emph{inverse Schreier automaton} and with a slight abuse of notation we denote it again by $\mathcal{A}$.
\begin{defn}
Let $\Gamma$ be a complete $A$-graph and let $H=\mathcal{G}(\Gamma)$ be the associated transition group. Consider $G$ to be a group with $H$ as a finite index subgroup and $\mathcal{A}$ the corresponding Schreier automaton over $A \cup T$. The \emph{inflated graph} $\mathcal{A} \ltimes (\Gamma,x)$ with $x \in V(\Gamma)$ is the rooted $A \cup T$-graph defined as follows:
the set of vertices is $T \times V(\Gamma)$, a generic edge is $(t,p) \mapright{y} (t',p')$ with $t \mapright{y \mid u} t' \in \mathcal{A}$ and $p \mapright{u} p'$ a walk in $\Gamma$ and the root is $(\mathds{1}_{G},x)$. 
\end{defn}
Intuitively, the inflated graph simulates a cascade composition of the Schreier automaton $\mathcal{A}$ with the graph $\Gamma$: each transition of $\mathcal{A}$ labelled by $y \mid u$ triggers the traversal of the word $u$ inside $\Gamma$. 
This construction is not genuinely new, it can be viewed as a composition of sequential machines (transducers) in disguise, where the graph $\Gamma$ itself is interpreted as a transducer. It is straightforward to see that $\mathcal{A} \ltimes \Gamma$ is inverse and if $\Gamma$ is context-free, so is $ \mathcal{A} \ltimes \Gamma$ (see the proof of \cite[Lemma 11]{RodContextfree}).
\begin{prop}\label{prop:closure_overgroups}
The class of {\bf CF-TR} is closed by taking finite index overgroups.
\end{prop}
\begin{proof}
Let $G$ be a finitely generated group and $H=\mathcal{G}(\Gamma)$ a finite index subgroup of $G$. Let $T$ be a transversal of $G$ with respect to $H$ containing the identity, so that $H=\langle A\rangle$ and $G= \langle A \cup T \rangle$, and let $\mathcal{A}$ be the corresponding Schreier automaton.\\
Let $\{ \Delta_{1}, \ldots, \Delta_{N}\}$ be the frontier vertices, and arguing as we did in Proposition~\ref{prop: f.g. subgroups} with a slight abuse of notation, we identify the frontier vertices of $\Delta_i$ with the frontier vertices of an end cone $\Gamma(v,x_0)$ of $\Gamma$ that is end-isomorphic to $\Gamma_i$. We claim that $G$ is the transition group of 
\[
\Theta_\mathcal{A}:=\bigcup_{i=1}^{N} \bigcup_{v \in \Delta_{i}}(\mathcal{A} \ltimes (\Gamma,v),(v,\mathds{1}_G)).
\]
All we need to prove is that $g \neq 1$ in $G$ if and only if there exists $i \in \{1, \ldots, N\}$, $v \in \Delta_i$ and $q \in \mathcal{A} \ltimes (\Gamma,v)$ s.t. $q \cdot g \neq q$. Indeed the latter is exactly the requirement for an element to be non trivial in $\mathcal{G}(\Theta_{\mathcal{A}})$. \\

First of all, we notice that $g$ is non trivial in $G \setminus T$ if and only if at least one of the following occurs
\begin{enumerate}
    \item there exists $t \in T$ s.t. $tg=_{G} ht'$ for some $t' \neq t$ and $h \in \{h_{t,y}\}_{t \in T,y \in A \cup T-\{\mathds{1}_G\}}^{+}$;
    \item for all $t \in T$ there exists $h \in H$ s.t. $g=_{G} ht$ and $h$ is non trivial in $H$.
\end{enumerate}

Suppose there exists $i$ and $v \in \Delta_i$ as in the hypothesis, then $(t,p)$ with $t \in T$ and $p \in V(\Gamma)$ is such that $(t,p) \cdot g=(t',p')$ and in $\mathcal{A}$ we have $t \mapright{g \mid h} t'$. If $t \neq t'$, then in $\mathcal{A}$ we have  $t \mapright{g \mid h} t' \neq t$ hence we have $(t,p) \cdot g= (t', p \cdot h) \neq (t,p)$ in $\mathcal{A} \rtimes (\Gamma,v)$ for some $i$ and $v \in \Delta_i$. Otherwise by definition we have $p \mapright{h} p \cdot h \neq p$ in $\Gamma$ and $h$ is non trivial in $H$. So, by Lemma~\ref{lem:infinite-path}, there exists $\widetilde{p} \mapright{\mathring{h}^{\omega}}\ldots \subseteq \Gamma(\widetilde{p},x_0)$ with $\mathring{h} \in S_A(h)$.
Therefore there exists $p' \in V(\Gamma,v)$ for some $i$ and $v \in \Delta_i$ such that $p' \cdot h \neq p'$ and $p' \mapright{h} p' \cdot h \subseteq \Gamma(v,x_0)$. 
Finally, in $\mathcal{A} \rtimes (\Gamma,v)$ we have $(\mathds{1}_G,v) \cdot u= (t,p')$ where $u$ is such that there exists $\mathds{1}_G \mapright{u \mid z} t$ in $\mathcal{A}$ with $v \mapright{z} p' \subseteq (\Gamma,v)$ and $(t,p') \cdot g =(t,p' \cdot h) \neq (t,p')$ since by hyphotesis  $tg=_G ht$ and $p' \cdot h \neq p'$.
\end{proof}

\begin{rem} \label{rmk:direct-product-CF}
This means that any finitely generated subgroup of a virtual direct product of free groups is {\bf CF-TR}. In particular, finitely generated virtually abelian groups are in {\bf CF-TR}.
\end{rem}

\section{Detecting torsion} \label{sec:torsion}
 We give some structural results focusing on the existence of torsion elements and on bounds on the order of such elements. Henceforth, we consider a context-free $A$-graph $\Gamma$ with $N$ end-cone types $\Gamma_1,\ldots, \Gamma_N$ with the corresponding frontier vertices $\Delta_1, \ldots, \Delta_N$ with $\Delta_i\subseteq \Gamma_i$. For any end-cone type $\Gamma_i$ that is end-isomorphic to some end-cone $\Gamma(v,x_0)$ with $\|v\|_{x_0}=n$ there are $k_i \le 2|A|$ second level end-cones $\Gamma(q_1, x_0), \ldots, \Gamma(q_{k_i}, x_0)$ contained in $\Gamma(v,x_0)$ with $\|q_j\|_{x_0}=n+1$. 
Roughly speaking we are taking the end-cones of $\Gamma_i$ whose frontier vertices are at distance one from $\Delta_i$. Put $\gamma=\max\{k_i, i\in [1,N]\}$ and $\delta=\max\{|\Delta_i|, i\in [1,N]\}$.\

We say that two end-cones $\Gamma(p,x_0), \Gamma(q,x_0)$ are \emph{nested} if $\Gamma(p,x_0)\subseteq \Gamma(q,x_0)$ and there is an end-isomorphism $\psi:\Gamma(q,x_0)\to \Gamma(p,x_0)$ with $\psi(q)=p$. We have the following lemma.
\begin{lemma}\label{lem: geodesic infinite order}
Let $\Gamma(p,x_0), \Gamma(q,x_0)$ be two distinct end-isomorphic nested end-cones of $\Gamma$. Then, for any walk $q\mapright{u} p\subseteq \Gamma(q,x_0)$ there is an infinite walk $q\mapright{u^{\omega}} \ldots\subseteq  \Gamma(q,x_0)$ that is not a circuit. 
\end{lemma}
\begin{proof}
Indeed, let $\psi:\Gamma(q,x_0)\to \Gamma(p,x_0)$ be the end-isomorphism with $\psi(q)=p$. Consider a geodesic $q\mapright{g}\psi(q)\subseteq \Gamma(q,x_0)$ connecting $q$ with $\psi(q)=p$. We claim that also $q\mapright{g}\psi(q)\mapright{g}\psi^2(q)$
is a geodesic. Any geodesic $q\mapright{u}\psi^2(q)\subseteq \Gamma(q,x_0)$ has to cross a vertex $v\in \Delta(\psi(q), x_0)$, thus the previous geodesic factors as $q\mapright{u_1}v\mapright{u_2}\psi^2(q)$ with $q\mapright{u_1}v\subseteq \Gamma(q,x_0)$ and $v\mapright{u_2}\psi^2(q)\subseteq \Gamma(\psi(q),x_0)$. Since both $q\mapright{g}\psi(q)$ and $\psi(q)\mapright{g}\psi^2(q)$ are geodesics, respectively in $\Gamma(q,x_0)$ and $\Gamma(\psi(q),x_0)$ we deduce $|u_1|=|g|$ and $|u_2|=|g|$. Hence, $q\mapright{g}\psi(q)\mapright{g}\psi^2(q)$ is also a geodesic. The same argument can be extended to show that $q\mapright{g^k}\psi^{k}(q)\subseteq \Gamma(q,x_0)$ is a geodesic for any $k\ge 1$. Now, by applying $\psi^k$ to the walk $q\mapright{u}\psi(q)\subseteq \Gamma(q,x_0)$, we deduce that $q\mapright{u^{\omega}}\subseteq \Gamma(q,x_0)$ is an infinite walk that is not a circuit.
\end{proof}
A subgraph $\Lambda$ of $\Gamma$ is called \emph{$g$-periodic} if $\Lambda$ is equal to the subgraph induced by the circuit $\pi=y\mapright{g^k}y$ for some integer $k\ge 1$. The \emph{order} $o(g,\Lambda)$ is the smallest integer $m$ such that $\Lambda$ is equal to the subgraph induced by the vertices and edges of the circuit $\pi=y\mapright{g^m}y$.
A \emph{clip} is a walk $y_1 \mapright{u} y_2\subseteq \Gamma(y_1, x_0)$ with $y_1, y_2 \in \Delta(y_1, x_0)$.
For a clip $p$, we denote by $(\chi(p), \subseteq)$ the poset of all subwalks of $p$ which are clips ordered by inclusion.
\begin{lemma}\label{lem: chain}
    For a clip $p=y_1 \mapright{u} y_2$, let $h(\chi(p))$ be the length of a maximal chain in $\chi(p)$. Then,
    $h(\chi(p))\ge \max\{d(y_1, x): x\in p\}+1$.
\end{lemma}
\begin{proof}
Let $k = \max\{ d(y_1, x) : x \in p \}$. We proceed by induction on $k$. The case $k = 0$ is trivial. Assume that $k > 0$. Let $q_1$ and $q_2$ denote, respectively, the first and the last vertices of $p$ such that $d(y_1, q_1) = d(y_1, q_2) = k - 1$.  Such vertices exist because, by the tree-like structure of $\Gamma$, the walk $p$ must pass through the set of frontier vertices at distance $k - 1$ both when entering and when exiting $\Gamma(q_1, x_0)$ which is a second level end-cone of $\Gamma(y_1, x_0)$.
Then the subwalk $p' = q_1 \mapright{u'} q_2$ of $p$ is a clip with $\max\{ d(q_1, x) : x \in p' \} = k - 1$. 
By the induction hypothesis, we have $h(\chi(p')) \ge k$. 
Hence, since $p' \subsetneq p$, it follows that $h(\chi(p)) \ge h(\chi(p')) + 1 \ge k + 1$.
\end{proof}
\begin{lemma}\label{lem: order of cyclic graph}
    For any $g$-periodic subgraph $\pi$ of a context-free graph $\Gamma$ we have $o(g,\pi)< \left(\frac{|g|-1}{|g|}\delta\right)\gamma^{\delta^2 N|g|}$.
\end{lemma}
\begin{proof}
Let $\pi$ be the subgraph induced by the walk $p = y \mapright{g^{\ell}} y$, where $\ell = o(g, \pi)$. 
Note that, without loss of generality, by cyclically shifting $g$ if necessary, we may assume that $p \subseteq \Gamma(y, x_0)$, so that $p$ is a clip.
For any clip $\lambda \in \chi(p)$ with $\lambda = y_1 \mapright{u} y_2$ and $y_1, y_2 \in \Delta(y_1, x_0)$, 
we say that $\lambda$ is \emph{colored} by $(x_1, x_2, \Gamma_i, g')$ if the following conditions hold:
\begin{itemize}
    \item $\lambda \subseteq \Gamma(y_1, x_0)$ and there exists an end-isomorphism 
    $\psi : \Gamma(y_1, x_0) \to \Gamma_i$ such that $x_1=\psi(y_1), x_2=\psi(y_2) \in \Delta_i$;
    \item $g'$ is a prefix of $g$, and there exists an integer $s \ge 0$ such that 
    $y \mapright{g^s g'} y_1$.
\end{itemize}
Let us prove first the following claim:
\begin{equation}\label{eq: claim}
\mbox{if } \lambda_1, \lambda_2\in \chi(p), \lambda_1\subsetneq\lambda_2, \mbox{then }\lambda_1, \lambda_2\mbox{ have different colors.}
\end{equation}
Suppose by contradiction that claim (\ref{eq: claim}) does not hold, and $\lambda_1, \lambda_2$ have the same color $(x_1. x_2, \Gamma_i, g')$. Therefore, $\lambda_1$ and $\lambda_2$ can be factorized as follows:
\[
\lambda_2 = y_1 \mapright{g''} q_1 \mapright{g^m} q_2 \mapright{g'} y_1' 
\mapright{g''u_2} y_2' \mapright{u_3} y_2, 
\qquad 
\lambda_1 = y_1' \mapright{g''u_2} y_2',
\]
for some $m \ge 0$, where $g'g'' = g$.  
Since $\lambda_1$ and $\lambda_2$ have the same color, there exists an end-isomorphism $\psi : \Gamma(y_1, x_0) \to \Gamma(y_1', x_0)$ between two nested end-cones $\Gamma(y_1', x_0) \subseteq \Gamma(y_1, x_0)$.  
Hence, by Lemma~\ref{lem: geodesic infinite order}, there exists an infinite walk $y_1 \vvlongmapright{(g'' g^m g')^{\omega}} \ldots = y_1 \vvlongmapright{(g'' g')^{\omega}} \ldots$ that is not a circuit, and this implies that $\pi$ is infinite,  a contradiction.
\\
By the determinism and the minimality of $\ell$ each vertex $q\in V(\pi)$ occurs at most $|g|-1$ times in the vertices appearing in the description of the walk $p = y \mapright{g^{\ell}} y$. Hence, the number of vertices appearing in $p$ is at most $(|g| - 1)\,|V(\pi)|$. 
We claim that $|V(\pi)| \le \delta \gamma^{C(g)}$, where $C(g)$ denotes the number of possible colors. 
Suppose, for the sake of contradiction, that $|V(\pi)| > \delta \gamma^{C(g)}$. 
Since the number of vertices in $\Gamma(y, x_0)$ at distance $n$ satisfies 
$|\Gamma(y) \cap D_n(y)| < \delta \gamma^n$, there must exist a vertex 
$v \in p$ with $d(v, y) \ge C(g)$. 
Then, by Lemma~\ref{lem: chain}, we conclude that the length of a maximal chain $\lambda_1 \subseteq \lambda_2 \subseteq \ldots \subseteq \lambda_m$ in $\chi(p)$ is strictly greater than $C(g)+1$, that is, $m > C(g)$. 
Therefore, by the pigeonhole principle, there exist $\lambda_i \subsetneq \lambda_j$ 
with the same color, contradicting our claim~\eqref{eq: claim}. 
The lemma now follows by noting that $C(g) \le \delta^2 N |g|$ and the inequality $|g|\ell\le (|g| - 1)\,|V(\pi)|$
\end{proof}
The following proposition provides information about the order of a element.
\begin{prop}\label{prop: order properties}
The following facts hold:
\begin{itemize}
\item let $\Gamma(p,x_0), \Gamma(q,x_0)$ be two distinct end-isomorphic nested end-cones of $\Gamma$, and let $q\mapright{u} p\subseteq \Gamma(q,x_0)\subseteq $ be a walk. Then, the element $u\in \mathcal{G}(\Gamma)$ has infinite order. 
\item Let $g\in \mathcal{G}(\Gamma)$ be an element of finite order, then the order $o(g)$ of the element $g$ satisfies $o(g)< e^{-N\delta}\delta^{N\delta|g|}\gamma^{\delta^3N^2|g|^2}=O(exp(|g|^2))$.
\end{itemize}
\end{prop}
\begin{proof}
The first claim follows from Lemma~\ref{lem: geodesic infinite order}: since $q \mapright{u^{\omega}}\ldots$ is an infinite walk that is not a circuit, we have $u^i \cdot q \neq u^j \cdot q$ for all $i \neq j$. 
Let us now prove the bound on the torsion element $g$. 
Note that $o(g)$ is upper bounded by the least common multiple of all possible values $o(g,\pi)$, where $\pi$ ranges over all $g$-periodic subgraphs. 
By Lemma~\ref{lem: order of cyclic graph}, these subgraphs are, up to end-isomorphism, finitely many; denote this finite family by $C = \{\pi_1, \ldots, \pi_k\}$. 
Each $\pi \in C$ is entirely contained in some end-cone $\Gamma_i$ and is the subgraph induced by the walk $y \mapright{g_1^{\ell}} y$, where $g_1$ is a cyclic shift of $g$. 
Since there are at most $|g|$ cyclic shifts of $g$, there can be at most $|g|$ $g$-periodic graphs $\pi \subseteq \Gamma_i$ with $y \in \Delta_i \cap V(\pi)$. 
Therefore, we may conclude that the cardinality $|C| \le |g| N \delta$. Hence, the least common multiple of all possible values of $o(g,\pi)$, $\pi\in C$, can be upper bounded by $(\frac{|g|-1}{|g|}\delta\gamma^{\delta^2 N|g|})^{|g| N \delta}$. Thus, we conclude $o(g)<(\frac{|g|-1}{|g|}\delta\gamma^{\delta^2 N|g|})^{|g| N \delta}\le e^{-N\delta}\delta^{N\delta|g|}\gamma^{\delta^3N^2|g|^2}$.
\end{proof}
\begin{rem}
Note that in the case of a finite union of context-free graphs 
$\{ \Gamma^{(1)}, \ldots, \Gamma^{(k)} \}$, 
the first part of the proposition clearly remains valid.
As for the second part, instead of considering the parameters $\delta$ and $N$, 
one must take
$
\max_{i \in \{1, \ldots, k\}} \delta_i$ and $
\max_{i \in \{1, \ldots, k\}} N_i,
$
where $\delta_i$ and $N_i$ denote the respective parameters of $\Gamma^{(i)}$.
\end{rem}

\begin{quest}
Is there a {\bf CF-TR} group realizing the upper bound provided in the second point of Proposition~\ref{prop: order properties}?
\end{quest}

As a consequence of the previous proposition we have:
\begin{prop}\label{prop: infiniteness}
Let $\{\Gamma^{(1)}, \ldots, \Gamma^{(k)}\}$ be a collection of context-free graphs, 
and let $G = \mathcal{G}\bigl(\Gamma^{(1)} \sqcup \cdots \sqcup \Gamma^{(k)}\bigr) \in \mathbf{CF\text{-}TR}$
be the associated transition group. Then, the following are equivalent:
\begin{enumerate}
    \item $G$ contains a copy of $\mathbb{Z}$;
    \item $G$ is infinite;
    \item at least one of the graphs $\Gamma^{(1)}, \ldots, \Gamma^{(k)}$ is infinite.
\end{enumerate}
\end{prop}
\begin{proof}
$(1)\Rightarrow (2)$ is trivial. For $(2)\Rightarrow (3)$, note that if all $\Gamma^{(i)}$ are finite, then $G$ is finite. For $(3)\Rightarrow (1)$, assume without loss of generality that $\Gamma^{(1)}$ is infinite. Since there are finitely many end-cone types but infinitely many end-cones, the pigeonhole principle guarantees the existence of at least two nested end-cones. By Proposition~\ref{prop: order properties}, this produces an element of infinite order.
\end{proof}

Another consequence of the previous facts is the following.
\begin{cor}\label{theo: torsion}
     An infinite group $\mathcal{G}\bigl(\Gamma^{(1)} \sqcup \cdots \sqcup \Gamma^{(k)}\bigr)\in {\bf CF-TR}$ does not contain an infinite finitely generated torsion subgroup. 
\end{cor}
\begin{proof}
    This follows from the closure property under taking finitely generated subgroups, as stated in Proposition~\ref{prop: f.g. subgroups} and Proposition~\ref{prop: infiniteness}.
\end{proof}

It remains an open problem whether the Grigorchuk group belongs to the class {\bf CoCF} (see \cite{CioEldFer}), and more generally, whether there are infinite torsion groups in the class {\bf CoCF}. However, using the previous corollary, we can show the following.
\begin{cor}\label{cor:grigorchuk}
No infinite torsion group can be {\bf CF-TR}. In particular, the Grigorchuk group is not {\bf CF-TR}.
\end{cor}
Proposition~\ref{prop: order properties} provides an alternative proof of the decidability of the torsion problem for the transition groups of context-free graphs by checking whether $g^k = \mathds{1}_G$ for some 
$k \le e^{-N\delta} \, \delta^{N\delta |g|} \, \gamma^{\delta^3 N^2 |g|^2}.$ 
This decidability result can also be deduced from the corresponding result for the class of co-context-free groups (see \cite[Theorem 9]{holt}).
\begin{theorem}
With the above notation, for a given element $g\in \mathcal{G}\bigl(\Gamma^{(1)} \sqcup \cdots \sqcup \Gamma^{(k)}\bigr)$ there is an algorithm to check whether $o(g)$ is finite or not. In other words, the torsion problem is solvable for {\bf CF-TR} groups.
\end{theorem}

We know that for $F$ a free group of rank at least $2$, the group $F \times F$ is in {\bf CF-TR}. Indeed, $F$ is in {\bf CF-TR} simply by considering the Cayley graph on free generators and, by Proposition \ref{prop:closure_products}, the class is closed under direct products. This provides an example of a {\bf CF-TR} group with unsolvable conjugacy problem (see e.g. \cite[Theorem 9, Ch.3]{miller}) and unsolvable generalized word problem (\cite[Theorem 4.6]{miller2}).\

\section{Poly-context-free groups and co-context-free transition groups} \label{sec:poly-cf-transition}
In this section, we address the first question of Problem~\ref{prob: general problems cf-trans-group} and, at the end, discuss some conjectures on poly-context-free groups. We first focus on groups whose word problem is the intersection of finitely many inverse-context-free languages (denoted {\bf Poly-ICF}), i.e., intersections of context-free languages accepted by graphs. The following theorem shows that this class is contained in {\bf CF-TR}; it corresponds to the special case of intersections of deterministic context-free languages, which are contained in {\bf Co-CF}.
\begin{theorem}\label{theo: poly with inverse graphs implies cf-trans}
Let $G$ be a group whose word problem $WP(G;A)$ is the intersection of $k$ context-free languages accepted by rooted graphs:
\[
WP(G;A)=\bigcap_{i=1}^k L(\Gamma^{(i)}, x_i)
\]
Then, $G\simeq \mathcal{G}(\Gamma^{(1)}\sqcup\ldots\sqcup \Gamma^{(k)})$ is in {\bf CF-TR}. 
\end{theorem}
\begin{proof}
Let $(\Gamma, \alpha) = \bigtimes_{i=1}^k (\Gamma^{(i)}, x_i)$ be the $A$-graph defined as the product of the rooted graphs $(\Gamma^{(i)}, x_i)$. The vertex set of $\Gamma$ is the Cartesian product $V(\Gamma) = V(\Gamma^{(1)}) \times \cdots \times V(\Gamma^{(k)})$, 
so each vertex is a $k$-tuple $(v_1, \ldots, v_k)$. A transition labeled by $a \in A$ acts componentwise:
\[
(v_1, \ldots, v_k) \xrightarrow{a} (v_1', \ldots, v_k') \quad \text{if and only if} \quad v_i \xrightarrow{a} v_i' \text{ in } \Gamma^{(i)} \text{ for all } i.
\]
The root of $\Gamma$ is $\alpha = (x_1, \ldots, x_k)$. Since each $\Gamma^{(i)}$ is inverse, it is immediate that $\Gamma$ is also inverse. Moreover, since $WP(G;A) = L(\Gamma, \alpha)$, Proposition~\ref{prop: inclusion implies homomorphism} implies that $\Gamma \simeq \Cay(G;A)$. Recall that for any $u \in A^*$ and $x \in V(\Gamma)$, $x \cdot u$ is the endpoint of the walk $x \xrightarrow{u} (x \cdot u)$ in a complete $A$-graph. Then $\alpha \cdot u = (x_1 \cdot u, \ldots, x_k \cdot u)$ and $(\Gamma, \alpha \cdot u) = \bigtimes_{i=1}^k (\Gamma^{(i)}, x_i \cdot u)$. 
Since $\Gamma$ is a Cayley graph, $(\Gamma, \alpha) \simeq (\Gamma, \alpha \cdot u)$ for all $u \in A^*$, and so $
WP(G;A) = \bigcap_{i=1}^k L(\Gamma, x_i \cdot u)$ for all $u \in A^*$. Using (\ref{eq:everywhere-circuits}), we obtain
\[
WP(G;A) = \bigcap_{i=1}^k \left( \bigcap_{y \in V(\Gamma^{(i)})} L(\Gamma^{(i)}, y) \right) = \bigcap_{i=1}^k \mathcal{L}(\Gamma^{(i)}) = \mathcal{L}(\Gamma^{(1)}, \ldots, \Gamma^{(k)}),
\]
i.e., $G \simeq \mathcal{G}(\Gamma^{(1)} \sqcup \cdots \sqcup \Gamma^{(k)})$.
\end{proof}
An interesting question is whether there exist examples of {\bf CF-TR} groups that are not poly-context-free. The next result shows that any such example must be non-quasi-transitive (for a specific example, see Section~\ref{sec: non-residually}). More precisely, it provides a characterization that can be seen as a kind of converse to the previous result.
\begin{theorem}\label{theo: virtually subgroups}
Let $\{\Gamma^{(1)},\ldots, \Gamma^{(k)}\}$ be a collection of quasi-transitive context-free $A$-graphs. Then, 
$G=\mathcal{G}(\Gamma^{(1)}\sqcup\ldots \sqcup \Gamma^{(k)})$ is a group in {\bf CF-TR} if and only if $G$ is virtually a finitely generated subgroup of the direct product of free groups. In particular, it belongs to the class of poly-context-free groups whose word problem is the intersection of finitely many languages accepted by rooted quasi-transitive context-free graphs.
\end{theorem}
\begin{proof}
Fix a generic $A$-graph $\Gamma^{(i)}$ of the family $\{\Gamma^{(1)}, \ldots, \Gamma^{(k)}\}$. Since it is quasi-transitive, there is a finite complete set $F_i\subseteq V(\Gamma^{(i)})$ of representatives of the equivalence relation given by the orbits of $\Aut(\Gamma_i)$. Since by Proposition~\ref{prop: inclusion implies homomorphism} we have that for any $\varphi\in\Aut(\Gamma^{(i)})$, $L(\Gamma^{(i)}, v)=L(\Gamma^{(i)}, \varphi(v))$
we deduce:
\[
\mathcal{L}(\Gamma^{(i)}) = \bigcap_{p\in V(\Gamma^{(i)})} L(\Gamma^{(i)}, p)=\bigcap_{\varphi\in \Aut(\Gamma^{(i)}), p\in F_i} L(\Gamma^{(i)}, \varphi(p))=\bigcap_{ p\in F_i} L(\Gamma^{(i)}, p)
\]
Since the word problem $WP(G;A)$ for $G=\mathcal{G}(\Gamma^{(1)}\sqcup\ldots \sqcup \Gamma^{(k)})$ is the intersection $\bigcap_{i=1}^k \mathcal{L}(\Gamma^{(i)})$ we have that $\mathcal{G}(\Gamma^{(1)}\sqcup\ldots \sqcup \Gamma^{(k)})$ is a poly-context-free group. In particular 
\[
WP(G;A)=\bigcap_{i=1}^k\left( \bigcap_{p\in F_i} L(\Gamma^{(i)}, p)\right)
\]
is the intersection of finitely many context-free languages accepted by quasi-transitive $A$-graphs, hence by \cite[Theorem~3]{RodContextfree} this is equivalent for the group $G$ to be virtually a finitely generated subgroup of the direct product of free groups.
Conversely, by \cite[Proposition~12]{RodContextfree} if $G$ is virtually a finitely generated subgroup of the direct product of free groups, then the word problem $WP(G;A)=\bigcap_{i=1}^k L(\Gamma^{(i)}, x_i)$ is the intersection of finitely many inverse-context-free languages accepted by graphs that are quasi-transitive. Hence, by Theorem~\ref{theo: poly with inverse graphs implies cf-trans} we have that the group $G\simeq \mathcal{G}(\Gamma^{(1)}\sqcup\ldots\sqcup\Gamma^{(k)} )$ is the transition group of the disjoint union of finitely many context-free $A$-graphs that are quasi-transitive.
 \end{proof}
The previous theorem provides another characterization of groups that are virtually finitely generated subgroups of direct products of free groups. In particular, \cite[Theorem~3.1]{RodContextfree} shows that this class coincides with groups whose word problem lies in {\bf Poly-QT-ICF}, i.e., intersections of finitely many languages accepted by graphs that are quasi-transitive and quasi-isometric to a tree. Note that the example in subsection~\ref{sec: example} belongs to {\bf Poly-QT-ICF}, since a quick inspection of Fig.~\ref{fig: omega} shows that $\Omega$ is quasi-transitive.

Following \cite{TaraHolt}, for $c = \langle c_0, \ldots, c_s \rangle \in \mathbb{Z}^{s+1}$ with $s \ge 1$, $c_0, c_s \ne 0$, and $\mathrm{GCD}(c_0,\ldots,c_s) = 1$, the group $G(c)$ is defined by 
\[
G(c) = \langle a,b \mid \mathfrak{R}_c \rangle, \quad \mathfrak{R}_c = \{ [b,b^{a^i}] \ (i \in \mathbb{Z}), \ b^{c_0} (b^a)^{c_1} \cdots (b^{a^s})^{c_s} \},
\] 
using the usual notation $x^y := y^{-1}xy$ and $[x,y] := x^{-1}y^{-1}xy$. These are called \emph{$Gc$-groups}, and a $Gc$-group is \emph{proper} if it is not virtually abelian.
We conclude this section with some remarks on Brough Conjecture~\ref{cj: Brough2}, which asserts that a finitely generated poly-context-free group is solvable if and only if it is virtually abelian. In this direction, we have the following proposition.
\begin{prop}\label{prop: solvable}
    If $G\in {\bf CF-TR}$ is a finitely generated transition group that is solvable, then one of the following holds:
    \begin{itemize}
        \item $G$ is virtually abelian; 
        \item $G$ contains the free abelian group $\mathbb{Z}^{\infty}$ of countably infinite rank.
        \item $G$  has arbitrarily large finite subgroups.
    \end{itemize}
\end{prop}
\begin{proof}
Combining results from \cite{TaraHolt}, Theorem 2.12 in \cite{Tara} shows that a solvable finitely generated group $G$ which is not virtually abelian either contains
\begin{enumerate}
\item $\mathbb{Z}^{\infty}$, or 
\item a proper $Gc$-group, or 
\item a finitely generated subgroup $H$ with an infinite normal torsion subgroup $U$, such that $H/U$ is either free abelian or a proper $Gc$-group.
\end{enumerate} 
By assumption the group $G$ is {\bf CF-TR} then it is {\bf CoCF} and so,
since by \cite[Proposition~4.13]{Tara} a $Gc$-group is poly-context-free or co-context-free if and only if it is virtually abelian, we deduce that case (2) is not possible.
Let us now analyze case (3) and assume that $G$ contains a subgroup $H$ with an infinite normal torsion subgroup $U$. Suppose that $U=\langle g_i: i\in\mathbb{N}\rangle$ and consider the subgroups $U_n=\langle g_i: i\le n\rangle \le U$. The subgroups $U_n$ are clearly finitely generated torsion subgroups of $G$ and thus, by Corollary~\ref{theo: torsion}, they are all finite. Clearly, the order of the subgroups $U_n$ becomes arbitrarily large since $U$ is infinite.
\end{proof}
In Brough Conjecture~\ref{cj: Brough3} it is also conjectured that a poly-context-free group does not have arbitrarily large finite subgroups. In the intersection of {\bf CF-TR} with poly-context-free groups we have the following fact. 
\begin{cor}
If it is true that a group $G$ which is both poly-context-free and {\bf CF-TR} does not have arbitrarily large finite subgroups, then the following are equivalent for a finitely generated group 
    \begin{enumerate}
        \item $G$ is virtually abelian
        \item $G$ is a solvable poly-context-free transition group in ${\bf CF-TR}$
    \end{enumerate}
\end{cor}
\begin{proof}
(1) $\implies$ (2)
 Since $\mathbb{Z}/n\mathbb{Z}$ and $\mathbb{Z}$ can be realized as transition groups of context-free graphs by Example~\ref{thm:example-Z-finite}, then Proposition~\ref{prop:closure_products} and Proposition~\ref{prop:closure_overgroups} imply that finitely generated virtually abelian groups are also in the class of transition groups as well as poly-context-free by \cite{Tara}.
(2) $\implies$ (1)
 The converse implication follows by Proposition~\ref{prop: solvable} and the fact that $\mathbb{Z}^{\infty}$ is not poly-context-free, as it contains free abelian subgroups of rank $k$ for all $k\in\mathbb{N}$, see \cite[Lemma~4.6]{Tara}.
\end{proof}
The previous discussion raises the following natural questions, which will be addressed in the following sections.
\begin{quest}\label{prob: z infinity and unbound finite groups}
\phantom{ }
\begin{itemize}
    \item Is it possible to realize a solvable transition group associated to a context-free graph containing $\mathbb{Z}^{\infty}$ as a subgroup? See Propositon~\ref{prop:solvable-z-infinito}.
    \item Is it true that a transition group associated to a context-free graph does not have arbitrarily large finite subgroups? By the argument at the end of the proof of Proposition~\ref{prop: solvable} this is equivalent to not having an infinite torsion subgroup. See the Section~\ref{subsec: perturbing} and Proposition~\ref{prop:boundary-infinite-torsion}.
\end{itemize}
\end{quest}

\section{Rationality of {\bf CF-TR}} \label{sec:rationality}
Before turning to the questions discussed in the previous section, we first examine a property common to all currently known examples of {\bf CoCF} groups and shared by our class.\
A \emph{sequential transducer} is a quintuple
$T=(Q,q_0, \Sigma, \delta, o)$, where
\begin{enumerate}
    \item  $Q$ is a finite set \emph{states};
    \item  $q_0 \in Q$ is the \emph{initial state};
    \item $\delta$ is a function $Q \times \Sigma \rightarrow Q$ called \emph{transition function};
    \item  $o$ is a function $Q \times \Sigma \rightarrow \Sigma^{*}$ called \emph{output function}.
\end{enumerate}
See, for instance, \cite{berstel}. For the sake of brevity, we denote $\delta (q,\sigma)=\widetilde{q}$ and $o(q,\sigma)=w$ simply with the diagram $q \longfr{\sigma}{w} \widetilde{q}$. It is straightforward that one can extend $o$ and $\delta$ to finite and infinite words. More precisely, if $\widehat{\Sigma}=\Sigma^{*} \cup \Sigma^{\omega}$, we have $\widehat{o}:Q \times \widehat{\Sigma} \rightarrow \widehat{\Sigma}$ and $\widehat{\delta}: Q \times \widehat{\Sigma} \rightarrow Q$ simply by the formula 
$$
\widehat{o}(q,\sigma_1 \sigma_2)=o(q,\sigma_1)o(\delta(q,\sigma_1),\sigma_2) 
\text{  and  }  
\widehat{\delta}(q,\sigma_1 \sigma_2)=\delta(\delta(q_1,\sigma_1),\sigma_2).
$$ 
Furthermore, fixed $q_0$, we have the functions $q_0 \ast: \widehat{\Sigma} \rightarrow \widehat{\Sigma}$ and $q_0 \cdot:Q \rightarrow Q$ with
$$
q_0 \ast w= \widehat{o}(q_0,w)
\text{  and  } 
q_0 \cdot w= \widehat{\delta}(q_0,w). 
$$
A homeomorphism $\varphi:\Sigma^{\omega} \rightarrow \Sigma^{\omega}$ is \emph{rational} if there exists a sequential transducer $(Q,q_0, \Sigma, \delta, o)$ such that $\varphi(\eta)= q_0 \ast \eta$ for all $\eta \in \Sigma^{\omega}$. 

\begin{defn}[\cite{GNS}]
The set of rational homeomorphisms over an alphabet $\Sigma$ is a group $\mathcal{R}_{\Sigma}$ (called the \emph{rational group}).
\end{defn}

Any two rational groups over two different alphabets are isomorphic (so it makes sense to simply write $\mathcal{R}$).

In \cite{BHM} it is shown that the rational group $\mathcal{R}$ is simple and non-finitely generated. We say that a group which embeds in any rational group is \emph{rational}. From a topological point of view $\Sigma^{\omega}$ is a Cantor space where the metric is given by the longest common prefix. In this way, there is a well-defined notion of bi-lipschitz homeomorphism. A transducer associated to a homeomorphism is said to be bi-lipschitz if the homeomorphism is bi-lipschtz. The family of groups which can be represented by bi-lipschitz transducers includes Thompson-like groups, self-similar groups and hyperbolic groups (see \cite{BelkBleakMatu}).
We want to construct a bi-lipschitz transducer for each element of our group belonging to {\bf CF-TR}, so that any such group embeds in $\mathcal{R}$. To fix the notation, let $\{\Gamma_{0}, \ldots, \Gamma_{N}\}$ be the representative of each end-isomorphim class of the context-free graph $(\Gamma, x_{0})$, and let $\{\Delta_{0}, \ldots, \Delta_{N}\}$ be the corresponding frontier vertices. We put $\Gamma_{0}=\Gamma$. As we have discussed at the end of Lemma~\ref{lem: geodesic infinite order}, for each end-cone $\Gamma(x, x_{0})$ of a context-free graph graph $(\Gamma, x_{0})$ we have a certain number of second-level end-cones which are end-isomorphic to some representative among $\{\Gamma_{0}, \ldots, \Gamma_{N}\}$. This means that each $\Gamma_{i}$ has $k(i)$ second-level end-cones that we denote by $\Gamma^{(i)}_{1}, \ldots, \Gamma^{(i)}_{k(i)}$. Each subgraph $\Gamma^{(i)}_{j}$ is end isomorphic to some end-cone, say $\Gamma_{h}\in \{\Gamma_{0}, \ldots, \Gamma_{n}\}$, and let us denote by $\Psi(i,j):\Gamma^{(i)}_{j}\to \Gamma_{h}$ such end-isomorphism and put $\psi(i,j)=h$, where $i\in [0,N]$, $j\in [1,k(i)]$. We also denote by $\Delta_{j}^{(i)}$ the frontier vertices of $\Gamma^{(i)}_{j}$. We now define the alphabet on which our transducer will act:
\[
\Lambda=\{(\Gamma_{j}^{(i)}, v): v\in\Delta_{j}^{(i)}, i\in[1,N], j\in [1,k(i)]\}\cup\{(\Gamma_{j}^{(i)}), i\in[1,N], j\in [1,k(i)]\}\cup \{(\Gamma_{0}), (\Gamma_{0}, x_{0})\}
\]
We say that a letter $\lambda\in\Lambda$ is \emph{final} if it is of the form $\lambda=(\Gamma_{j}^{(i)}, v)$; in this case we put $\lambda(1)=\Gamma_{j}^{(i)}$, $\lambda(2)=v$, while for a non-final letter $\lambda=(\Gamma_{j}^{(i)})$ we put $\lambda(1)=\Gamma_{j}^{(i)}$, $\lambda(2)=0$. Moreover, for any $\Gamma_{j}^{(i)}, \Gamma_{k}^{(h)}$ we say that $\Gamma_{j}^{(i)}<\Gamma_{k}^{(h)}$ if $\Gamma_{k}^{(h)}$ is a second-level end-cone of $\Psi(i,j)(\Gamma_{j}^{(i)})$. Note that in this case we have $h=\psi(i,j)$. We extend this relation of the letters of $\Lambda$ by putting $\lambda_{1}<\lambda_{2}$ if $\lambda_{1}(1)<\lambda_{2}(1)$. A word $w\in \Lambda^{*}$ is called \emph{well-formed}, if $w=\lambda_{0}\ldots \lambda_{\ell}$ with $\lambda_{0}=(\Gamma_{0})$, $\lambda_{\ell}$ is the only final letter of $w$, and $\lambda_{0}<\lambda_{1}<\ldots <\lambda_{\ell}$; we denote by $\mathcal{F}(\Lambda)$ the set of well-formed words on $\Lambda$. Roughly speaking, a well-formed word $w$ represents a vertex  $x\in\Gamma$ and the sequence $\lambda_{0}(1)<\lambda_{1}(1)<\ldots <\lambda_{\ell}(1)$ represents the sequence of end-cones traversed by a geodesic connecting $x_{0}$ to $x$. More precisely, an \emph{end-geodesic} is a sequence of end-cones $\Gamma(x_{0},x_{0})\supset \Gamma(x_{1}, x_{0})\ldots \supset \Gamma(x_{\ell}, x_{0}) $ such that $\|x_{i+1}\|_{x_{0}}=\|x_{i}\|_{x_{0}}+1$ for $i\in [0, \ell-1]$. Note that each $ \Gamma(x_{i+1}, x_{0})$ is a second-level end-cone of $ \Gamma(x_{i}, x_{0})$. A \emph{rooted end-geodesic} is given by end-geodesic $\Gamma(x_{0},x_{0})\supset \Gamma(x_{1}, x_{0})\ldots \supset \Gamma(x_{\ell}, x_{0}), x_{\ell})$ where we fix a vertex $x_{\ell}$ in the last end-cone. Henceforth, when we write $\Gamma(x_{0},x_{0})\supset \Gamma(x_{1}, x_{0})\ldots \supset \Gamma(x_{\ell}, x_{0}) $ we always mean a rooted end-geodesic with root given by the last vertex $x_{\ell}$ in the sequence. We have the following lemma.
\begin{lemma}\label{lem: well-formed}
Let $w=(\Gamma_{0})(\Gamma_{j_{1}}^{(i_{1})})\ldots (\Gamma_{j_{s}}^{(i_{s})})\ldots (\Gamma_{j_{\ell-1}}^{(i_{\ell-1})})(\Gamma_{j_{\ell}}^{(i_{\ell})}, v_{\ell})\in\mathcal{F}(\Lambda)$ be a well-formed word in $\Lambda^{*}$. Then, there is a unique rooted end-geodesic $\Gamma(x_{0},x_{0})\supset \Gamma(x_{1}, x_{0})\ldots \supset \Gamma(x_{\ell}, x_{0}) $ and end-isomorphisms 
\[
\varphi_{s}: \Gamma(x_{s}, x_{0})\to \Gamma_{j_{s}}^{(i_{s})},\, s\in [1,\ell]
\]
where $\varphi_{0}$ is the identity map, $\varphi_{\ell}(x_{\ell})=v_{\ell}$, and satisfying $(\Psi(i_{s-1}, j_{s-1})\circ \left.\varphi_{s-1})\right|_{\Gamma(x_{s}, x_{0})}=\varphi_{s}$. Vice versa, for any rooted end-geodesic $\Gamma(x_{0},x_{0})\supset \Gamma(x_{1}, x_{0})\ldots \supset \Gamma(x_{\ell}, x_{0})$ there is a unique well-formed word satisfying the previous properties. 
\end{lemma}
\begin{proof}
By induction on the parameter $s$. If $s=0$, then $\varphi_{0}$ the identity satisfies the desired conditions. Since $w$ is well-formed, $\Gamma_{j_{1}}^{(i_{1})}$ is a second-level end-cone that is a subgraph of $\Gamma_{0}$. In this case, necessarily $\Gamma(x_{1}, x_{0})=(\Gamma_{j_{1}}^{(i_{1})}, x_{1})$ and we put $\varphi_{1}=\left.\varphi_{0}\right|_{\Gamma(x_{1}, x_{0})}$. Let us assume that we have proved that there is an rooted end-geodesic 
\[
\Gamma(x_{0},x_{0})\supset \Gamma(x_{1}, x_{0})\ldots \supset \Gamma(x_{s-1}, x_{0})
\]
and end-isomorphisms $\varphi_{i}$, $i\in [1,s]$, satisfying the statement of the lemma. Now, consider the end-isomorphism $\Psi(i_{s-1}, j_{s-1}):\Gamma_{j_{s-1}}^{i_{s-1}}\to \Gamma_{h}$, $h=\psi(i_{s-1}, j_{s-1})$ and consider the end-isomorphism $\phi_{s}= \Psi(i_{s-1}, j_{s-1})\circ \varphi_{s-1}:\Gamma(x_{s-1}, x_{0})\to \Gamma_{h}$. Since $w$ is well-formed, $\Gamma_{j_{s}}^{(i_{s})}$ is a second-level end-cone of $\Gamma_{h}$. Therefore, there is a second-level end-cone $\Gamma(x_{s}, x_{0})=(\phi_{s}^{-1}(\Gamma_{j_{s}}^{(i_{s})}), \phi_{s}^{-1}(v_{s}))$ of $\Gamma(x_{s-1}, x_{0})$, for some $v_{s}\in\Delta_{j_{s}}^{(i_{s})}$. Now, by putting $x_{s}= \phi_{s}^{-1}(v_{s})$ and $\varphi_{s}=\left.\phi_{s}\right|_{\Gamma(x_{s}, x_{0})}$ we have that $\Gamma(x_{s}, x_{0})$ is the unique second-level end-cone of $\Gamma(x_{s-1}, x_{0})$ satisfying $\varphi_{s}(x_{s})=v_{s}$ and $(\Psi(i_{s-1}, j_{s-1})\circ \left.\varphi_{s-1})\right|_{\Gamma(x_{s}, x_{0})}=\varphi_{s}$.\

For the other direction, let $\Gamma(x_{0},x_{0})\supset \Gamma(x_{1}, x_{0})\ldots \supset \Gamma(x_{\ell}, x_{0}) $ be a rooted end-geodesic. Let $\varphi_{0}$ be the identity map, and assume that we have constructed a word 
\[
(\Gamma_{0})(\Gamma_{j_{1}}^{(i_{1})})\ldots (\Gamma_{j_{s}}^{(i_{s})})\ldots (\Gamma_{j_{\ell-1}}^{(i_{\ell-1})}, v_{\ell-1})
\]
and end-isomorphisms $\varphi_{i}$, $i\in [0,\ell-1]$ satisfying the statement of the lemma. Then, let $\Gamma_{j_{\ell}}^{(i_{\ell})}$ to be the second-level end-cone $\Psi(i_{\ell-1}, j_{\ell-1})\circ \varphi_{\ell-1}(\Gamma(x_{\ell}, x_{0}))$ of $\Gamma_{j_{\ell-1}}^{(i_{\ell-1})}$, and $v_{\ell}=\Psi(i_{\ell-1}, j_{\ell-1})\circ \varphi_{\ell-1}(x_{\ell})$. Then, 
\[
(\Gamma_{0})(\Gamma_{j_{1}}^{(i_{1})})\ldots (\Gamma_{j_{s}}^{(i_{s})})\ldots (\Gamma_{j_{\ell-1}}^{(i_{\ell-1})})(\Gamma_{j_{\ell}}^{(i_{\ell})}, v_{\ell})
\]
satisfies the statement of the lemma, and the associated rooted end-geodesic is exactly $\Gamma(x_{0},x_{0})\supset \Gamma(x_{1}, x_{0})\ldots \supset \Gamma(x_{\ell}, x_{0}) $.
\end{proof}
\noindent
By the previous Lemma~\ref{lem: well-formed} for each well-formed word 
\[
w=(\Gamma_{0})(\Gamma_{j_{1}}^{(i_{1})})\ldots (\Gamma_{j_{s}}^{(i_{s})})\ldots (\Gamma_{j_{\ell-1}}^{(i_{\ell-1})})(\Gamma_{j_{\ell}}^{(i_{\ell})}, v_{\ell})\in\mathcal{F}(\Lambda)
\]
we may associate a vertex $x_{\ell}=\varphi_{\ell}^{-1}(v_{\ell})$, and vice versa, given a vertex $x_{\ell}$ in $\Gamma$, there is a unique rooted end-geodesic  $\Gamma(x_{0},x_{0})\supset \Gamma(x_{1}, x_{0})\ldots \supset \Gamma(x_{\ell}, x_{0}) $ for which we may associate a unique well-formed word that we denote by $w[x_{\ell}]$. Therefore, there is a one to one correspondence between vertices of $\Gamma$ and well-formed words. For any word $w\in\Lambda^{*}$, let $w\Lambda^{\omega}$ be the clopen set formed by the infinite words starting with the word $w$. By, Lemma~\ref{lem: well-formed} we have the following fact.
\begin{lemma}\label{lem: bijection}
The map $\Xi: V(\Gamma)\to \mathcal{F}(\Lambda)\Lambda^{\omega}$, defined by $\Xi(x)=w[x]\Lambda^{\omega}$, $\forall x\in V(\Gamma)$, is a bijection.
\end{lemma}
\noindent
This is a key point that will be used in proving that the action of each letter $a\in A$ on $\Gamma$ is rational. 
We now define the sequential transducer $\mathcal{T}_{\Gamma}$. This transducer acts on $\Lambda$ and has set of states 
\[
Q_{\Gamma}=\left\{(\Gamma_{j}^{(i)},a): i\in[1,N], j\in [1,k(i)],  a\in A\right\}\cup\left\{(\Gamma_{0}, a): a\in A\right\}\cup\{E\}
\]
We now describe the set of transitions. The state $E$ is a sink state, which means that we have the transitions $E\mapright{\lambda\mid \lambda}E$ for all $\lambda\in \Lambda$. For each fixed $a\in A$ we now describe a transducer that mimics the action of the letter $a$ when we are in a vertex $x$ encoded in the closed set $\Xi(x)$. This is done by first checking whether the word in the input encodes a well-formed word representing a vertex $x$, and then changing the last letters to match the new vertex $x\cdot a$ in the new encoding in the output. The transitions are defined as follows.
 We start by describing how the initial state acts:
\begin{equation}\label{eq: initial}
\begin{cases}
(\Gamma_{0}, a)\mapright{\lambda|1}(\Gamma_{0}, a)& \mbox{if }\lambda=(\Gamma_{0})\\
(\Gamma_{0}, a)\mapright{\lambda|\lambda}E& \mbox{ in the other cases when }\lambda \mbox{ is not final. }
\end{cases}
\end{equation}
where the case when $\lambda$ is final is treated later in (\ref{eq: stay}) and (\ref{eq: new}). 
The transition of a generic state $(\Gamma_{j}^{(i)}, a)$ with input $\lambda$ depends on whether $\lambda(1)=\Gamma_{s}^{(h)}<\Gamma_{j}^{(i)}$ and whether $\lambda$ is final or not. Let us first assume $\lambda(1)=\Gamma_{s}^{(h)}<\Gamma_{j}^{(i)}$. Now, if $\lambda$ is not final we have the following transition:
\begin{equation}\label{eq: not final}
(\Gamma_{j}^{(i)}, a)\longfr{\lambda}{\lambda'}(\lambda(1), a), \mbox{ where }\lambda'=(\Gamma_{j}^{(i)}),  \mbox{\bf  if }\lambda(2)=0\mbox{ and }\lambda(1)<\Gamma_{j}^{(i)};\\
\end{equation}
in case it is final, we are in three cases (\ref{eq: get back}), (\ref{eq: stay}), (\ref{eq: new}) described below and the action depends on the end-cone $\lambda(1)=\Gamma_{s}^{(h)}$ and the vertex $\lambda(2)$. In the first case, when applying the letter $a$ to the vertex $\lambda(2)$ we get back to the previous end-cone:
\begin{align}\label{eq: get back}
&(\Gamma_{j}^{(i)}, a)\longfr{\lambda}{\lambda'}E, \lambda'=(\Gamma_{j}^{(i)}, z),  \mbox{ \bf  if }\lambda(2)\neq 0, \Gamma_{s}^{(h)}<\Gamma_{j}^{(i)},\\& \mbox{ and }\lambda(2)\mapright{a}\Psi(i,j)(z) \subseteq \Psi(i,j)(\Gamma_{j}^{(i)}) \mbox{ is an edge in } \Psi(i,j)(\Gamma_{j}^{(i)})\mbox{ with }z\in\Delta_{j}^{(i)}. \notag
\end{align}
The second case is when applying the letter $a$ to the vertex $\lambda(2)$ we are staying in the same frontier vertices of the end-cone $\lambda(1)=\Gamma_{s}^{(h)}$:
\begin{align}\label{eq: stay}
& (\Gamma_{j}^{(i)}, a)\vlongfr{\lambda}{\lambda'\lambda''}E, \mbox{ where }\lambda'=(\Gamma_{j}^{(i)}),\lambda''=(\Gamma_{s}^{(h)}, z),  \mbox{ \bf if }\lambda(2)\neq 0, \Gamma_{s}^{(h)}<\Gamma_{j}^{(i)}  \\
&\mbox{ and }\lambda(2)\mapright{a}z\subseteq \Gamma_{s}^{(h)} \mbox{ is an edge in }\Gamma_{s}^{(h)},\mbox{ with }z\in\Delta_{s}^{(h)} \notag.
\end{align}
The last case occurs when applying the letter $a$ to the vertex $\lambda(2)$ we are moving to a frontier vertex on a second-level end-cone $\Gamma_{t}^{(m)}$ of $\lambda(1)=\Gamma_{s}^{(h)}$:
\begin{align}\label{eq: new}
&(\Gamma_{j}^{(i)}, a)\vvvlongfr{\lambda}{\lambda'\lambda''\lambda'''}E, \mbox{ where }\lambda'=(\Gamma_{j}^{(i)}),\lambda''=(\Gamma_{s}^{(h)}), \lambda'''=(\Gamma_{t}^{(m)}, z)   \mbox{ \bf if }\lambda(2)\neq 0,\\
&\Gamma_{t}^{(m)}<\Gamma_{s}^{(h)}<\Gamma_{j}^{(i)}, \mbox{ and }\Psi(s,h)(\lambda(2))\mapright{a}z \mbox{ is a walk in }\Psi(s,h)(\Gamma_{s}^{(h)})=\Gamma_{m},\mbox{ with }z\in\Delta_{t}^{(m)}. \notag
\end{align}
If the previous condition $\lambda(1)=\Gamma_{s}^{(h)}<\Gamma_{j}^{(i)}$ does not hold, we add the transition
\begin{equation}\label{eq: not well formed}
(\Gamma_{j}^{(i)}, a)\vlongfr{\lambda}{\lambda'\lambda}E, \lambda'=(\Gamma_{j}^{(i)})\mbox{ whether }\lambda\mbox{ is final or not}.
\end{equation}
We will prove that each state $(\Gamma_{0}, a)$ represents the action of the letter $a\in A$ of each vertex of $\Gamma$. We have the following lemma.
\begin{lemma}\label{lem: well-formed fixed}
For any word $w=\lambda_{0}\ldots \lambda_{k}\in\Lambda^{*}$ formed by non-final letters, with $\lambda_{0}=(\Gamma_{0})$ and satisfying $\lambda_{0}<\ldots <\lambda_{k}$ we have:
\[
(\Gamma_{0}, a)\ast w=\lambda_{0}\ldots \lambda_{k-1}, \quad (\Gamma_{0}, a)\cdot w=(\lambda_{k}(1), a)
\]
for any letter $a \in A$.
\end{lemma}
\begin{proof}
It follows by induction on the length of $w$, where the base of induction is a consequence of (\ref{eq: initial}) while the general case uses the transitions described in (\ref{eq: not final}).
\end{proof}
\begin{lemma}\label{lem: not well-formed}
For any infinite word $\eta\in\Lambda^{\omega}\setminus\left(\mathcal{F}(\Lambda)\Lambda^{\omega}\right)$ that is not well-formed, we have $(\Gamma_{0}, a)\ast \eta=\eta$ for any letter $a \in A$.
\end{lemma}
\begin{proof}
Let $w=\lambda_{0}\ldots \lambda_{k}$ be a (possibly empty) maximal prefix of $\eta$ formed by non-final letters, with $\lambda_{0}=(\Gamma_{0})$ and $\lambda_{0}<\ldots <\lambda_{k}$. If $w$ is empty, this means that $\lambda_{0}\neq (\Gamma_{0}), (\Gamma_{0}, x_{0})$, hence by the transition described in (\ref{eq: initial}) we deduce $(\Gamma_{0}, a)\ast \eta=\eta$. Otherwise, $w$ is non-empty, and if $k$ is finite, then $\lambda_{k}\nless \lambda_{k+1}$. By Lemma~\ref{lem: well-formed fixed} we have $(\Gamma_{0}, a)\ast w=\lambda_{0}\ldots \lambda_{k-1}$ and $(\Gamma_{0}, a)\cdot w=(\lambda_{k}(1),a)$. Hence, using the condition $\lambda_{k}\nless \lambda_{k+1}$ and by transition (\ref{eq: not well formed}) we conclude that $(\Gamma_{0}, a)\ast w\lambda_{k+1}=w\lambda_{k+1}$ and $(\Gamma_{0}, a)\cdot w\lambda_{k+1}=E$, i.e. $(\Gamma_{0}, a)\ast \eta=\eta$.
If $k$ is infinite, meaning $w=\eta$, by Lemma~\ref{lem: well-formed fixed} we may conclude our claim $(\Gamma_{0}, a)\ast \eta=\eta$.
\end{proof}
We recall that for a vertex $x\in \Gamma$ and a letter $a\in A$ we denote by $x\cdot a$ the vertex such that $x\mapright{a} (x\cdot a)$ is an edge in $\Gamma$. We have the following lemma. 
\begin{lemma}\label{lem: transducer action}
Let $a\in A$ and $x\in V(\Gamma)$. Then, $(\Gamma_{0}, a)\ast w[x]\eta=w[x\cdot a]\eta$ for all $\eta\in\Lambda^{\omega}$.
\end{lemma}
\begin{proof}
Let
\[
w[x]=\lambda_{0}\ldots\lambda_{\ell}=(\Gamma_{0})(\Gamma_{j_{1}}^{(i_{1})})\ldots (\Gamma_{j_{s}}^{(i_{s})})\ldots (\Gamma_{j_{\ell-1}}^{(i_{\ell-1})})(\Gamma_{j_{\ell}}^{(i_{\ell})}, v_{\ell})
\]
be the well-formed word associated to the (unique) rooted end-geodesic $\Gamma(x_{0},x_{0})\supset \Gamma(x_{1}, x_{0})\ldots \supset \Gamma(x_{\ell}, x_{0}) $ with $x_{\ell}=x$, for which there are  end-isomorphisms $\varphi_{s}: \Gamma(x_{s}, x_{0})\to \Gamma_{j_{s}}^{(i_{s})}$ with $\varphi_{\ell}(x_{\ell})=v_{\ell}$ 
and satisfying $(\Psi(i_{s-1}, j_{s-1})\ast \left.\varphi_{s-1})\right|_{\Gamma(x_{s}, x_{0})}=\varphi_{s}$ for $s\in [1,\ell]$. By Lemma~\ref{lem: well-formed fixed} we have 
\begin{equation}\label{eq: equality}
(\Gamma_{0}, a)\ast \lambda_{0}\ldots \lambda_{\ell-1}=\lambda_{0}\ldots\lambda_{\ell-2},\quad (\Gamma_{0}, a)\cdot w=(\lambda_{\ell-1}(1), a)
\end{equation}
Now, since $\lambda_{\ell}$ is final and $\lambda_{\ell-1}<\lambda_{\ell}$ we are in one of the three possibile transitions (\ref{eq: get back}), (\ref{eq: stay}), (\ref{eq: new}). Suppose that $\lambda_{\ell}=\Gamma_{s}^{(h)}$, and $\lambda_{\ell-1}=\Gamma_{j}^{(i)}$, where $\Gamma_{s}^{(h)}$ is a second-level end-cone of $\Psi(i,j)(\Gamma_{i}^{(j)})=\Gamma_{h}$. 
\begin{itemize}
\item Suppose that condition (\ref{eq: get back}) holds. In this case $v_{\ell}\mapright{a}\Psi(i,j)(z)$ is an edge of $\Gamma_{h}$ with $z\in \Delta_{j}^{(i)}$, and so by (\ref{eq: equality}) and by the definition we get:
\[
(\Gamma_{0}, a)\ast w\eta=\lambda_{0}\ldots\lambda_{\ell-2}(\Gamma_{j}^{(i)},z)\eta,\quad \forall \eta\in\Lambda^{\omega}
\]
We claim that $\Xi(x\cdot a)=\lambda_{0}\ldots\lambda_{\ell-2}(\Gamma_{j}^{(i)},z)\Lambda^{\omega}$. Now, since $\Psi(i,j)^{-1}(v_{\ell})\mapright{a} z$ is an edge of $\Gamma_{j}^{(i)}$ and $\varphi_{\ell}(x)=\varphi_{\ell-1}(x)$,  we deduce that also $x\mapright{a}\varphi^{-1}_{\ell-1}(z)$ is an edge in $\Gamma(x_{\ell-1}, x_{0})$ with $x\cdot a=\varphi^{-1}_{\ell-1}(z)\in\Delta(x_{\ell-1}, x_{0})$. Therefore, by Lemma~\ref{lem: well-formed}, the well-formed word $\lambda_{0}\ldots\lambda_{\ell-2}(\Gamma_{j}^{(i)},z)$ is associated the rooted end-geodesic:
\[
\Gamma(x_{0},x_{0})\supset \Gamma(x_{1}, x_{0})\ldots \supset \Gamma(x_{\ell-2}, x_{0})\supset \Gamma(x\cdot a, x_{0})
\]
and this concludes the proof of our claim.
\item Suppose that condition (\ref{eq: stay}) holds. In this case we have 
\[
(\Gamma_{0},  a)\ast w\eta=\lambda_{0}\ldots\lambda_{\ell-1}(\Gamma_{j_{\ell}}^{(i_{\ell})}, z)\eta,\quad \forall \eta\in\Lambda^{\omega}
\]
where $v_{\ell}\mapright{a}z$ is an edge of $\Gamma_{j_{\ell}}^{(i_{\ell})}$ with $z\in \Delta_{j_{\ell}}^{(i_{\ell})}$. Now, since $\varphi_{\ell}(v_{\ell})=x$ and $\varphi^{-1}_{\ell}(v_{\ell})\mapright{a}\varphi^{-1}_{\ell}(z)$ is an edge in $\Gamma(x_{\ell}, x_{0})$, we conclude that $x\cdot a=\varphi^{-1}_{\ell}(z)$. Hence, $\lambda_{0}\ldots\lambda_{\ell-1}(\Gamma_{j_{\ell}}^{(i_{\ell})}, z)$ is the well-formed word associated to the rooted end-geodesic $\Gamma(x_{0},x_{0})\supset \Gamma(x_{1}, x_{0})\ldots \supset \Gamma(x_{\ell-1}, x_{0})\supset \Gamma(x\cdot a, x_{0})$, i.e., $(\Gamma_{0}, a)\ast w[x]\eta=w[x\cdot a]\eta$ for all $\eta\in\Lambda^{\omega}$.
\item Suppose that condition (\ref{eq: new}) holds. Using the definition, we have
\[
(\Gamma_{0}, a)\ast w\eta=\lambda_{0}\ldots\lambda_{\ell-1}(\Gamma_{j_{\ell}}^{(i_{\ell})})(\Gamma_{t}^{(m)},z)\eta,\quad \forall \eta\in\Lambda^{\omega}
\]
where $\Gamma_{t}^{(m)}<\Gamma_{s}^{(h)}=\Gamma_{j_{\ell}}^{(i_{\ell})}$ and $\Psi(s,h)(v_{\ell})\mapright{a}z$ is an edge in $\Gamma_{m}$ with $z\in \Delta_{t}^{(m)}$.  Since $\Gamma_{t}^{(m)}$ is a second-level end-cone of $\Gamma_{m}$ and $\phi= \Psi(s,h)\circ\varphi_{\ell} :\Gamma(x,x_{0})\to \Gamma_{m}$ is an end-isomorphism, we deduce that $\phi^{-1}(\Gamma_{t}^{(m)})$ is a second level end-cone of $\Gamma(x, x_{0})$ containing $\phi^{-1}(z)$, and since $x\mapright{a}\phi^{-1}(z)$ is an edge in $\Gamma(x,x_{0})$ we have $\phi^{-1}(\Gamma_{t}^{(m)})=\Gamma(x\cdot a, x_{0})$. Take $\varphi_{\ell+1}=\left.\phi\right|_{\Gamma(x\cdot a, x_{0})}$. We clearly have $(\left.\Psi(s,h)\circ\varphi_{\ell})\right|_{\Gamma(x\cdot a, x_{0})}=\varphi_{\ell+1}$. Thus, by Lemma~\ref{lem: well-formed} we have that the rooted end-geodesic 
\[
\Gamma(x_{0},x_{0})\supset \Gamma(x_{1}, x_{0})\ldots \supset \Gamma(x_{\ell}, x_{0})\supset \Gamma(x\cdot a, x_{0})
\]
 is associated to the well-formed word $\lambda_{0}\ldots\lambda_{\ell-1}(\Gamma_{j_{\ell}}^{(i_{\ell})})(\Gamma_{t}^{(m)},z)$ and this is our claim  $(\Gamma_{0}, a)\ast w[x]\eta=w[x\cdot a]\eta$ for all $\eta\in\Lambda^{\omega}$.
\end{itemize}
\end{proof}
\begin{theorem}\label{theo:rational}
The transition group $\mathcal{G}\bigl(\Gamma^{(1)} \sqcup \cdots \sqcup \Gamma^{(k)}\bigr)$ of a collection of context-free complete graphs $\Gamma^{(1)}\ldots, \Gamma^{(k)}$ is rational.
\end{theorem}
\begin{proof}
Since rationality is closed by taking direct products and pass to subgroups, by Lemma~\ref{lem:subgroup-product} we reduce to the case $\mathcal{G}(\Gamma)$.\

As a shorthand notation let us put $\gamma_{a}=(\Gamma_{0},a)$, $a\in A$. By Lemma~\ref{lem: transducer action} and the fact that $\Gamma$ is a graph, we have $(\gamma_{a^{-1}}\gamma_{a})\ast w[x]\eta=w[x]\eta$ and $(\gamma_{a}\gamma_{a^{-1}})\ast w[x]\eta=w[x]\eta$ for all $\eta\in\Lambda^{\omega}$. Moreover, by Lemma~\ref{lem: not well-formed} for all $\xi\in\Lambda^{\omega}\setminus\mathcal{F}(\Lambda)$ we have $(\gamma_{a^{-1}}\gamma_{a})\ast \xi=\xi$ and $(\gamma_{a}\gamma_{a^{-1}})\ast \xi=\xi$. Therefore, we may conclude that each state $\gamma_a$ of $\mathcal{T}_{\Gamma}$ gives rise to a rational bijection $\gamma_a\ast : \Lambda^{\omega}\to \Lambda^{\omega}$ with inverse $\gamma_{a^{-1}}\ast$. Thus, by \cite[Proposition~1.3]{BelkBleakMatu}, each state $\gamma_a$ is a rational homeomorphism of $\Lambda^{\omega}$. Let us prove that the subgroup $H(\Gamma)=\langle \gamma_a: a\in A\rangle$ of the rational group $\mathcal{R}_{\Lambda}$ of all the rational homeomorphisms of $\Lambda^{\omega}$ is isomorphic to the transtion group $\mathcal{G}(\Gamma)$. Consider a word $u=a_{1}\ldots a_{p}\in A^{*}$. By Lemma~\ref{lem: transducer action}, and using a simple induction on the length $p$ of a word $u$ we have that the following equality
\[
(\gamma_{a_{k}}\ldots \gamma_{a_{1}})\ast w[x]\eta=w[x\cdot u]\eta,\,\forall \eta\in\Lambda^{\omega}
\]
holds for any vertex $x\in \Gamma$. It is clear that if $\gamma_{a_{k}}\ldots \gamma_{a_{1}}$ represents the identity in $H(\Gamma)$, then 
\[
w[x]\eta=(\gamma_{a_{k}}\ldots \gamma_{a_{1}})\ast w[x]\eta=w[x\cdot u]\eta.
\]
Now, since the well-formed words form a prefix code, then for an infinite word $\eta$ there is at most one (maximal) prefix that is well-formed, by the previous equality we conclude that $w[x]=w[x\cdot u]$ for all $x\in V(\Gamma)$, i.e., $u\in \mathcal{L}(\Gamma)$ represents the identity in $\mathcal{G}(\Gamma)$. Therefore, the isomorphism $\phi:\{\gamma_a:a\in A\}^*\to A$ defined by $\phi(\gamma_a)=a$ extends to the epimomorphism $\varphi:H(\Gamma)\to \mathcal{G}(\Gamma)$ naturally. To prove that $\varphi$ is an isomorphism suppose that $\varphi(\gamma_{a_{k}}\ldots \gamma_{a_{1}})=u$ represents the identity in $\mathcal{G}(\Gamma)$, then $x\cdot u=x$ for all $x\in\Gamma$, that is equivalent to  $w[x]=w[x\cdot u]$ for all $x\in \Gamma$ which in turn is equivalent to $(\gamma_{a_{k}}\ldots \gamma_{a_{1}})\ast w[x]\eta=w[x]\eta$ for all $\eta\in\Lambda^{\omega}$, $x\in V(\Gamma)$. Now, if $\xi$ is not well-formed, by Lemma~\ref{lem: not well-formed} we have $(\gamma_{a_{k}}\ldots \gamma_{a_{1}})\ast \xi=\xi$. Therefore, we have shown that $(\gamma_{a_{k}}\ldots \gamma_{a_{1}})\ast \xi=\xi$ for all $\xi\in\Lambda^{\omega}$, i.e., $\gamma_{a_{k}}\ldots \gamma_{a_{1}}$ represents the identity in $H(\Gamma)$. This concludes the proof that $H(\Gamma)\simeq \mathcal{G
}(\Gamma)$, i.e., $\mathcal{G}(\Gamma)$ embeds into the rational group $\mathcal{R}_\Lambda$.
\end{proof}

Note that, despite we already have a tree structure associated to the graph, namely the tree of end-cones, the transducer acts on $\Lambda^{\omega}$ which contains a copy of such tree as a subtree.

\begin{rem}
The transducer that represents any element $g$ in a {\bf CF-TR} group is bi-lipschitz. It suffices to notice that any element in $\mathcal{F}(\Lambda)$ is shortened or lengthened by at most $|g|$ letters.
\end{rem}


\section{Perturbing a graph: a structural result}\label{subsec: perturbing}
Another natural question about these groups is how much small local changes in the graph affect the associated transition groups. 
In this section we investigate what happens to the transition group when we locally perturb a context-free graph $(\Gamma, x_{0})$. The idea is to consider a disk $D_{n}(x_{0})$ and to associate with the graph $\Lambda_{n}=\Gamma\setminus D_{n}(x_{0})$ a group called the boundary group $\mathrsfs{B}_{n}$ at level $n$, and then to give a structural result connecting this group to the transition group $G=\mathcal{G}(\Gamma)$. 
\begin{defn}[Boundary group]
Consider the following subset of $A^{*}$
\[
\mathcal{O}_{n}=\{u\in A^{*}: \mbox{ if }x\mapright{u}x'\mbox{ is a walk in }\Lambda_{n}\mbox{ then }x=x'\}
\]
of the words that either label a circuit totally contained in $\Lambda_{n}$, or whose walk intersects $S_{n}(x_{0})$. 
It is easy to verify that $\mathcal{O}_{n}$ is a submonoid of $A^{*}$ satisfying the property that for any $u\in A^{*}$ and $w\in \mathcal{O}_{n}$ we have $u^{-1}wu\in \mathcal{O}_{n}$. Therefore, if $\pi:A^{*}\to G$ is the natural map, the image $\mathrsfs{B}_{n}=\pi(\mathcal{O}_{n})$ is a normal subgroup of $G$ called the \emph{boundary group at level} $n$.
\end{defn}
We have the following key lemma.
\begin{lemma}
With the above notation, for any $u \in\mathcal{O}_{n} $ and $x\in V(\Gamma)$ we have $x\cdot u^{k}=x$ for some $k\le \left |S_{n}(x_{0})\right|\,\left |u\right |$.
\end{lemma}
\begin{proof}
Consider the infinite walk $x\mapright{u^{\omega}}\ldots$, and let us denote by $x_{i}^{(j)}$ the vertices of such walk with $x=x_{0}^{(0)}$ and $x_{i}^{(j)}$ for $0 \le i \le |u|$ are all the vertices between $x_{0}^{(j)}$ and $x_{0}^{(j+1)}$ so that $ x_{0}^{(j)} \mapright{u} x_{0}^{(j+1)}$
are the corresponding subwalks of $x\mapright{u^{\omega}}\ldots $. If $x\cdot u=x$ we are done, thus we may suppose $x\cdot u\neq x$. Note that every subwalk $x_{0}^{(j)}\mapright{u} x_{0}^{(j+1)}$ intersects $S_{n}(x_{0})$. Indeed, if $x_{0}^{(j)}\mapright{u} x_{0}^{(j+1)}$ were totally contained in $\Lambda_{n}$, then since $u\in \mathcal{O}_{n}$, we would have $x_{0}^{(j)}=x_{0}^{(j+1)}$, hence $x\cdot u=x$ when $j=0$, a contradiction. 
Now if we have that $x_{i}^{{(j)}}= x_{i}^{{(s)}}=v \in S_{n}(x_{0})$, for some $s\neq j$, then by the determinism of $\Gamma$ we have $x\cdot u^{k}=x$ for some $k$. 
Since the walk $x\mapright{u^{\omega}}\ldots$ is infinite and $S_{n}(x_{0})$ is a finite set. Either one of the two cases happens.
To give the upper bound on $k$, we consider the worst case scenario: each subwalk $x_{0}^{(j)}\mapright{u} x_{0}^{(j+1)}$ contains at most a vertex of $S_{n}(x_{0})$, therefore by the pigeonhole principle it is easy to see that condition $x_{i}^{{(j)}}= x_{i}^{{(s)}} $ occurs for some $j,s\le  \left |S_{n}(x_{0})\right|\,\left |u\right |$. 
\end{proof}
As an immediate consequence of the previous lemma, we have the following corollary.
\begin{cor}\label{cor: torsion perturbation}
For all $n\ge 1$, the boundary group $\mathrsfs{B}_{n}$ at level $n$  is a torsion normal subgroup of $G$ where each element $g$ has order $o(g)\le \left |S_{n}(x_{0})\right|\,\left |g\right |$ (where $|\cdot|$ is the group metric with respect to $A$).
\end{cor}
In view of the structural result contained in Proposition~\ref{theo: virtually subgroups} an interesting case is when the graph $\Lambda_{n}$ can be obtained by a quasi-transitive graph by cutting out the disk $D_{n}(x_{0})$, more precisely we define the following class of graphs.
\begin{defn}
A graph $(\Gamma, x_{0})$ is called \emph{locally-quasi-transitive} if there is a quasi-transitive graph $(\Theta, y_{0})$ and an integer $n\ge 1$ such that $\Gamma\setminus D_{n}(x_{0})=\Lambda_{n}=\Theta\setminus D_{n}(y_{0})$.
\end{defn}
In case all the orbits are infinite we may explicitly calculate the set $\mathcal{O}_{n}$.
\begin{lemma}\label{lem: infinite orbits}
Suppose the quasi-transitive graph $\Theta$ has finitely many orbits $O_{1},\ldots, O_{m}$ which are all infinite. Let $R=\{p_{1}, \ldots, p_{m}\}$ be a set of representative vertices such that $p_{i} \in O_i$ for $1 \leq i \leq m $.
Then,
\[
\mathcal{O}_{n}=\bigcap_{p\in R}L(\Theta, p)=\mathcal{L}(\Theta)
\]
\end{lemma}
\begin{proof} 
The equality $\bigcap_{p\in R}L(\Theta, p)=\mathcal{L}(\Theta)$ follows by Proposition~\ref{prop: inclusion implies homomorphism} as two vertices in the same orbits have the same associate language.
Since $\bigcap_{p\in R}L(\Theta, p)=\mathcal{L}(\Theta)$ are words that label a circuit at every point, we clearly have the inclusion $\bigcap_{p\in R}L(\Theta, p)\subseteq \mathcal{O}_{n}$. Conversely, take a word $u\in \mathcal{O}_{n}$, for each $p_{i}\in R$ we claim that $u$ labels a circuit at $p_{i}$. Indeed, since the orbit $O_{i}$ is infinite, we may find a vertex $p_{i}'\in O_{i}$ that is at sufficiently large distance from the disk $D_{n}(x_{0})$ such that $p_{i}'\mapright{u}z$ is totally contained in $\Lambda_{n}$, hence by definition $z=p_{i}'$ and therefore $u$ labels also a circuit at $p_{i}$. Therefore, $u\in \bigcap_{p\in R}L(\Theta, p)$ showing the other inclusion. 
\end{proof}
In case the graph $(\Gamma, x_{0})$ is locally-quasi-transitive where the associated quasi-transitive graph $\Theta$ has all the orbits that are infinite, we say that $(\Gamma, x_{0})$ is \emph{locally-quasi-transitive with infinite orbits}. We are now in position to prove the following structural theorem.
\begin{theorem}\label{theo: structural locally quasi-trans}
Let $(\Gamma, x_{0})$ be a context-free graph that is also locally-quasi-transitive with infinite orbits and let $G= \mathcal{G}(\Gamma)$ be its transition group. Then, there is a short exact sequence
\[
\mathds{1}\rightarrow \mathrsfs{B}_{n}\rightarrow G\rightarrow H\rightarrow \mathds{1}
\]
where $H$ is a group that is virtually a finitely generated subgroup of the direct product of free groups and $\mathrsfs{B}_{n}$ is a torsion normal subgroup where the order of each element $g$ is bounded above by $o(g)\le \left |S_{n}(x_{0})\right|\,\left |g\right |$. In particular we have $H\simeq G/_{\mathrsfs{B}_{n}}$.
\end{theorem}
\begin{proof}
Let $\Theta$ be the graph which is quasi-transitive with infinite orbits, such that there exists an integer $n\ge 1$ with $\Gamma\setminus D_{n}(x_{0})=\Lambda_{n}=\Theta\setminus D_{n}(y_{0})$. Since there are finitely many end-cone types with frontier vertices contained inside the disk of radius $n$, we have that $\Gamma$ and $\Theta$ differ from finitely many end-cones up to end-isomorphism. Hence, $\Theta$ is also context-free. Therefore, by Proposition~\ref{theo: virtually subgroups} we have that $H=\mathcal{G}(\Theta)$ is virtually a finitely generated subgroup of the direct product of free groups. Now, by Lemma~\ref{lem: infinite orbits} we have $\mathcal{O}_{n}=\mathcal{L}(\Theta)$, hence since  $\mathcal{L}(\Gamma)\subseteq \mathcal{O}_{n}=\mathcal{L}(\Theta)$,
there is a well-defined epimorphism $\psi: \mathcal{G}(\Gamma)\rightarrow \mathcal{G}(\Theta)$, with $\ker(\psi)=\pi(\mathcal{O}_{n})=\mathrsfs{B}_{n}$ ($\pi$ is the natural map from $A^{*}$ to the group $\mathcal{G}(\Gamma)$). From which we have the short exact sequence in the statement and the fact that $H\simeq G/_{\mathrsfs{B}_{n}}$. 
Moreover, by Corollary~\ref{cor: torsion perturbation}, we have that the boundary group $\mathrsfs{B}_{n}$ is a torsion subgroup of $\mathcal{G}(\Gamma)$ where the order of each element $g$ is bounded above by $o(g)\le \left |S_{n}(x_{0})\right|\,\left |g\right |$. 
\end{proof}
It is worth mentioning that in the settings of the previous proposition, the boundary group $\mathrsfs{B}_{n}$ is generated by a poly-context-free language.
This setting is going to be a key ingredient in Section~\ref{sec: non-residually} were we will address the second point of Question~\ref{prob: z infinity and unbound finite groups}.

\section{Free product of context-free graphs} \label{sec:free-product} 
The aim of this section is to introduce the operation of free product of graphs and proving that the property of being context-free is preserved under this operation. This is useful in constructing examples of context-free graphs from simpler ones. 
\begin{figure}[htp]
\includegraphics[scale=1.5]{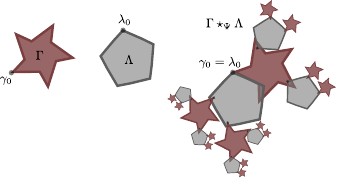}
\caption{A depiction of the free product of two rooted graphs $(\Gamma, \gamma_{0}), (\Lambda, \lambda_{0})$ with gluing maps $\psi^{1}(v)=\lambda_{0}$ for all $v\in\Gamma\setminus\{\gamma_{0}\}$ and $\psi^{2}(v)=\gamma_{0}$ for all $v\in\Lambda\setminus\{\lambda_{0}\}$.} \label{fig: free product}
\end{figure}
\begin{defn}[Free product of graphs]\label{def: free product}
Let $\Theta_1=\{(\Gamma^{(i)}, \gamma^{i}_0), i=1,\ldots, k\}$ be a collection of $A_1$-graphs, and similarly let $\Theta_2=\{(\Lambda^{(j)}, \lambda^{(j)}_0), j=1,\ldots, m\}$ be a collection of $A_2$-graphs, with $A_1\cap A_2=\emptyset$. For each $i\in[1,k]$, $j\in [1,m]$ we consider the functions $\psi^{1}_{i}:V(\Gamma^{(i)})\setminus\{\gamma_0^{i}\}\to \{\lambda^{1}_0,\ldots, \lambda^{m}_0\}$ and $\psi^{2}_{j}:V(\Lambda^{(j)})\setminus\{\lambda_0^{j}\}\to \{\gamma^{1}_0,\ldots, \gamma^{k}_0\}$; the gluing map is the collection $\Psi$ of all such functions. We first consider the infinite disjoint union of copies of the graphs $\Theta_1, \Theta_2$ where the set of vertices are words alternating between vertices of the graphs $\Theta_1$ and $\Theta_2$. More precisely, we consider the $A_1\cup A_2$-graph $(\Theta_1\sqcup\Theta_2)^{\oplus}$ whose set of vertices $\mathcal{V}$ is formed by alternating sequences of vertices of $\Theta_1, \Theta_2$ which are either of the form $\gamma^{i_1}_{s_1}\lambda^{j_1}_{t_1}\gamma^{i_2}_{s_2}\lambda^{j_2}_{t_2}\ldots \gamma^{i_\ell}_{s_\ell}\lambda^{j_\ell}_{t_\ell}, \mbox{ or } \lambda^{j_1}_{t_1}\gamma^{i_1}_{s_1}\lambda^{j_2}_{t_2}\gamma^{i_2}_{s_2}\ldots \lambda^{j_\ell}_{t_\ell}\gamma^{i_\ell}_{s_\ell}$, with $\gamma^{i_n}_{s_n}\in V(\Gamma^{(i_n)}), \lambda^{j_n}_{t_n}\in V(\Lambda^{(j_n)})$, $n\in [1,\ell]$ and $\ell \in \mathbb{N}$. The edges are of the form 
\[
\begin{cases}
\gamma^{i_1}_{s_1}\lambda^{j_1}_{t_1}\gamma^{i_2}_{s_2}\lambda^{j_2}_{t_2}\ldots \gamma^{i_\ell}_{s_\ell}\lambda^{j_\ell}_{t_\ell}\mapright{a}\gamma^{i_1}_{s_1}\lambda^{j_1}_{t_1}\gamma^{i_2}_{s_2}\lambda^{j_2}_{t_2}\ldots \gamma^{i_\ell}_{s_\ell}\lambda^{j_\ell}_{t_h}, a\in A_2, \mbox{ whenever }\lambda^{j_\ell}_{t_\ell}\mapright{a}\lambda^{j_{\ell}}_{t_h}\in E(\Lambda^{(j_{\ell})})\\
 \lambda^{j_1}_{t_1}\gamma^{i_1}_{s_1}\lambda^{j_2}_{t_2}\gamma^{i_2}_{s_2}\ldots \lambda^{j_\ell}_{t_\ell}\gamma^{i_\ell}_{s_\ell} \mapright{a}  \lambda^{j_1}_{t_1}\gamma^{i_1}_{s_1}\lambda^{j_2}_{t_2}\gamma^{i_2}_{s_2}\ldots \lambda^{j_\ell}_{t_\ell}\gamma^{i_\ell}_{s_h}, a\in A_1, \mbox{ whenever } \gamma^{j_\ell}_{s_\ell}\mapright{a}\gamma^{j_{\ell}}_{s_h}\in E(\Gamma^{(j_{\ell})})
\end{cases}
\]
Roughly speaking we are taking infinitely many copies of the graphs $\Gamma^{(i)},\Lambda^{(j)}$ where we simply rename each vertex by adding a prefix that is a word that is alternating in the sense explained above. 
Now, the free product $\Theta_1 \star_{\Psi}\Theta_2$ is the $A_1\cup A_2$-graph obtained from $(\Theta_1\sqcup\Theta_2)^{\oplus}$ by identifying the vertices of $\mathcal{V}$ via the transitive closure of the following relation:
\[
\gamma_0^{1}v^*=\lambda_0^{1}v^*, \;v^*\gamma_i^{j}u^*=v^*\psi_j^{1}(\gamma_i^{j})u^*, \; v^*\lambda_i^{j}u^*=v^*\psi_j^{2}(\lambda_i^{j})u^*
\]
where $u^{*}, v^{*}$ are generic alternating words on $\Theta_{1}, \Theta_{2}$ possibly empty. Henceforth, we will consider the rooted graph $(\Theta_1 \star_{\Psi}\Theta_2, \gamma_0^1)$ where $ \gamma_0^1= \lambda_0^{1}$.
\end{defn}
Roughly speaking the rooted graph $(\Theta_1 \star_{\Psi}\Theta_2, \gamma_0^1)$ is obtained from the graphs $\Gamma^{(1)}\sqcup \Lambda^{(1)}$ by first identifying $\gamma_0^1=\lambda_0^1$, and then by iteratively gluing at a vertex with representative $v^*\gamma^{j}_i$, with $i\neq 0$ (not a root), the graph $(\Lambda^{(s)}, \lambda^s_0)$ with $\lambda^s_0=\psi^1_j(\gamma^{j}_i)$, so that the two vertices $v^{*}\gamma_{i}^{j},v^{*}\lambda_{0}^{s}$ identify. In case the vertex ends with a vertex from the $\Lambda$ graphs such as $v^*\lambda^{j}_i$, then we glue the graph $(\Gamma^{(s)},\gamma_{0}^{s})$ by making the identification $\psi^{2}_{j}(\lambda^{j}_i)=\gamma_{0}^{s}$. Let us call a vertex $v^*\gamma^{j}_i$, where the last vertex belongs to some graph of the collection $\Theta_{1}$, a vertex colored by $1$, while if it ends by a vertex from the collection $\Theta_{2}$ we say that it is colored by $2$. To have an idea of what kind of graph we may generate see Fig.~\ref{fig: free product} in which the free product of two rooted graphs $(\Gamma, \gamma_{0}), (\Lambda, \lambda_{0})$ is depicted, where the gluing maps are $\psi^{1}(v)=\lambda_{0}$ for all $v\in\Gamma\setminus\{\gamma_{0}\}$ and $\psi^{2}(v)=\gamma_{0}$ for all $v\in\Lambda\setminus\{\lambda_{0}\}$. We will see later that if $\Gamma,\Lambda$ are two Cayley graphs of two groups $G_{1}, G_{2}$, then $(\Gamma, \gamma_{0})\star_{\Psi}(\Lambda, \lambda_{0})$ is the Cayley graph of the free product $G_{1}\star G_{2}$, hence the name.
The following proposition is an easy consequence of the definition and the condition $A_{1}\cap A_{2}=\emptyset$.
\begin{prop}\label{prop: free product and completeness}
If $\Theta_{1}$ ($\Theta_{2}$) is a family of deterministic (complete) $A_{1}$-graphs ($A_{2}$-graphs), then $\Theta_1 \star_{\Psi}\Theta_2$ is a deterministic (complete) $A_{1}\cup A_{2}$-graph. In particular, if $\Theta_{1}$ ($\Theta_{2}$) is a family of $A_{1}$-graphs ($A_{2}$-graphs), then the rooted graphs $(\Theta_1 \star_{\Psi}\Theta_2, \gamma_0^1)$ is a $A_{1}\cup A_{2}$-graph.
\end{prop}
We now make some geometric considerations on this construction. For simplicity of notation let $\alpha=\gamma_{0}^{1}=\lambda_{0}^{1}$ and put $\Omega= \Theta_1 \star_{\Psi}\Theta_2$. Now, consider a geodesic $\alpha\mapright{w}v^{*}\gamma_{i}^{j}$ connecting $\alpha$ to some vertex $v^{*}\gamma_{i}^{j}$ of $(\Omega, \alpha)$, colored by $1$ with $i\neq 0$. Since $i\neq 0$, $w$ ends with a word colored by $1$, thus it may be uniquely factor as an alternating product $w=z_{1}x_{1}z_{2}x_{2}\ldots z_{k}x_{k}$ with $z_{t}\in A_{2}$ and $x_{t}\in A_{1}$. Let $\Lambda^{(s)}$ be the graph of $\Theta_{2}$ with $\lambda_{0}^{s}=\psi^{1}_{j}(\gamma_{i}^{j})$ so that in $\Omega$, the vertex $v^{*}\gamma_{i}^{j}$ is identified with $v^{*}\lambda_{0}^{s}$. Note that we are gluing at $v^{*}\gamma_{i}^{j}$ the rooted $A_{1}\cup A_{2}$-graph $\Omega^{+}(v^{*}\gamma_{i}^{j})$ obtained from $\Lambda^{(s)}$ in which we iteratively glue all the graphs from the collection $\Theta_{1}, \Theta_{2}$ according to the gluing maps $\Psi$. More formally, this subgraph is isomorphic to the $A_{1}\cup A_{2}$-graph $\Lambda^{(s)}(\uparrow_{\Psi})$ which is the connected component containing $\lambda_{0}^{s}$ of the graph obtained from $(\Theta_1\sqcup\Theta_2)^{\oplus}$ by identifying the vertices of $\mathcal{V}$ via the equivalence relation generated by the relation
\[
 v^*\gamma_i^{j}u^*=v^*\psi_j^{1}(\gamma_i^{j})u^*, \; v^*\lambda_i^{j}u^*=v^*\psi_j^{2}(\lambda_i^{j})u^*
\]
Note that $\Lambda^{(s)}(\uparrow_{\Psi})\simeq \Omega^{+}(v^{*}\gamma_{i}^{j})$ via the map on the set of vertices defined by $\phi(\lambda_{0}^{s})=v^{*}\gamma_{i}^{j}=v^{*}\lambda_{0}^{s}$ and $\phi(\lambda_{0}^{s}u^{*})=v^{*}\lambda_{0}^{s}u^{*}$. The graph $\Omega^{+}(v^{*}\gamma_{i}^{j})$ is called the \emph{out-cone} of $\Omega$ at $v^{*}\gamma_{i}^{j}$. Moreover, since we are gluing a copy of $\Lambda^{(s)}(\uparrow_{\Psi})$ by identifying the vertex $\lambda_{0}^{s}$ with $v^{*}\gamma_{i}^{j}$, this means that all the geodesics $\alpha\mapright{w}v^{*}\gamma_{i}^{j}\mapright{w'}v^{*}\gamma_{i}^{j}u^{*}$ of $\Omega$
passing through $v^{*}\gamma_{i}^{j}$ with $w'=u_{1}u_{2}$, for some $u_{1}\in A_{2}^{+}$, have the property that $v^{*}\gamma_{i}^{j}u^{*}$ belongs to the out-cone $\Omega^{+}(v^{*}\gamma_{i}^{j})$, see Fig.~\ref{fig: outcone} for a graphical representation. We record this remark in the following lemma.
\begin{lemma}\label{lem: out-cones and geodesics}
Put $(\Omega, \alpha)=(\Theta_1 \star_{\Psi}\Theta_2, \gamma_0^1)$, then for any $\gamma_{i}^{j}, \lambda_{i}^{j}$ with $i\neq 0$ we have:
\begin{align*}
V(\Omega^{+}(v^{*}\gamma_{i}^{j}))=\{v^{*}\gamma_{i}^{j}\}\cup \{v^{*}\gamma_{i}^{j}u^{*}: \mbox{ there is a geodesic }\alpha\mapright{w}v^{*}\gamma_{i}^{j}\mapright{w'}v^{*}\gamma_{i}^{j}u^{*}\mbox{ with }w'\in A_{2}^{+}(A_{1}\cup A_{2})^{*} \}\\
V(\Omega^{+}(v^{*}\lambda_{i}^{j}))=\{v^{*}\lambda_{i}^{j}\}\cup\{v^{*}\lambda_{i}^{j}u^{*}: \mbox{ there is a geodesic }\alpha\mapright{w}v^{*}\lambda_{i}^{j}\mapright{w'}v^{*}\lambda_{i}^{j}u^{*}\mbox{ with }w'\in A_{1}^{+}(A_{1}\cup A_{2})^{*} \}
\end{align*}
Moreover, for any vertex $v^{*}\gamma_{i}^{j}=v^{*}\lambda_{0}^{s}$ colored by $1$, or $v^{*}\lambda_{i}^{j}=v^{*}\gamma_{0}^{s}$ colored by $2$, we have:
\[
\Omega^{+}(v^{*}\gamma_{i}^{j})\simeq \Lambda^{(s)}(\uparrow_{\Psi}), \quad \Omega^{+}(v^{*}\lambda_{i}^{j})\simeq \Gamma^{(h)}(\uparrow_{\Psi}).
\]
\end{lemma}
\begin{figure}[ht]
\includegraphics[scale=2]{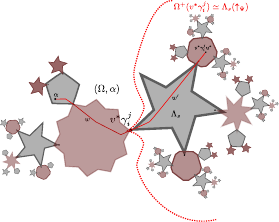}
\caption{In the dashed region the out-cone $\Omega^{+}(v^{*}\gamma_{i}^{j})$, in red the geodesic $\alpha\mapright{w}v^{*}\gamma_{i}^{j}\mapright{w'}v^{*}\gamma_{i}^{j}u^{*}$.} \label{fig: outcone}
\end{figure}
We now describe how to obtain the end-cones of $\Omega$ from the end-cones of the graphs in $\Theta_{1}\cup \Theta_{2}$. We have the following lemma which, for simplicity, we state for a generic vertex colored by $2$ but it clearly holds for any vertex of $\Omega$. 
\begin{lemma}\label{lem: end-cones of free product}
With the above notation, for a given vertex $v^{*}\gamma_{i}^{j}\lambda^{s}_{t}$ of $\Omega$ colored by $2$, the end-cone $\Omega(v^{*}\gamma_{i}^{j}\lambda^{s}_{t},\alpha)$ is isomorphic to the $A_{1}\cup A_{2}$-graph obtained by taking the disjoint union
\[
\Lambda^{(s)}(\lambda^{s}_{t}, \lambda^{s}_{0})\cup\biguplus_{\substack{z\in V(\Lambda^{(s)}(\lambda^{s}_{t}, \lambda^{s}_{0}))\\ \gamma^{h}_{0}=\psi^{2}_{s}(z)}} \Gamma^{(h)}(\uparrow_{\Psi})
\]
and then by identifying each vertex $z\in V(\Lambda^{(s)}(\lambda^{s}_{t}, \lambda^{s}_{0}))$ with $\gamma^{h}_{0}=\psi^{2}_{s}(z)$.
\end{lemma}
\begin{proof}
Set $n=\|v^{*}\gamma_{i}^{j}\lambda^{s}_{t}\|_{\alpha}$.
Then, by Lemma~\ref{lem: out-cones and geodesics} we may conclude that the end-cone $\Omega(v^{*}\gamma_{i}^{j}\lambda^{s}_{t},\alpha)$ is contained in the out-cone $\Omega^{+}(v^{*}\gamma_{i}^{j})\simeq \Lambda^{(s)}(\uparrow_{\Psi})$. Moreover, the frontier vertices of $\Omega(v^{*}\gamma_{i}^{j}\lambda^{s}_{t},\alpha)$ are in one to one correspondence with the frontier vertices of $\Lambda^{(s)}(\lambda_{t}^{s}, \lambda_{0}^{s})$. Therefore, $\Omega(v^{*}\gamma_{i}^{j}\lambda^{s}_{t},\alpha)$ may be obtained from $\Lambda^{(s)}(\lambda_{t}^{s}, \lambda_{0}^{s})$ by gluing at each vertex $v$ of $\Lambda^{(s)}(\lambda_{t}^{s}, \lambda_{0}^{s})$ the corresponding out-cone $\Gamma^{(h)}(\uparrow_{\Psi})$ by identifying $z$ with $\gamma_{0}^{h}=\psi^{2}_{s}(z)$. See Figure~\ref{fig: end-cones} for a graphical representation on how to obtain $\Omega(v^{*}\gamma_{i}^{j}\lambda^{s}_{t},\alpha)$ from $\Lambda^{(s)}(\lambda_{t}^{s}, \lambda_{0}^{s})$. 
\end{proof}
\begin{figure}[ht]
\includegraphics[scale=1.3]{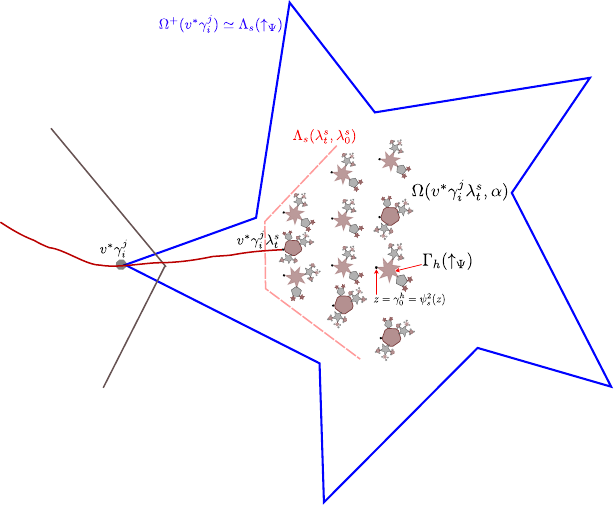}
\caption{A depiction of the end-cone $\Omega(v^{*}\gamma_{i}^{j}\lambda^{s}_{t},\alpha)$ which is obtained by gluing at each vertex $z$ of the end-cone $\Lambda^{(s)}(\lambda^{s}_{t}, \lambda^{s}_{0})$ of $\Lambda^{(s)}$ (in dashed red lines) the out-cone $\Gamma^{(h)}(\uparrow_{\Psi})$ according to the value of gluing map $\gamma_{0}^{h}=\psi_{s}^{2}(z)$.} \label{fig: end-cones}
\end{figure}

\subsection{Context-free graphs and free product}
In this section we prove that the free product of directed graphs preserves the property of being context-free. For this purpose, we need to slightly adapt the notion of isomorphism and end-cone of an $A_{1}$-graph to include the notion of vertex coloring. A \emph{vertex-coloring of a graph} $\Gamma$ is just a map $\chi:V(\Gamma)\to C$ onto some set of colors $C$. An \emph{isomorphism} $\psi:\Gamma^{(1)}\to \Gamma^{(2)}$ between two $A$-graphs with color maps $\chi_{1}:V(\Gamma)\to C$, $\chi_{2}:V(\Gamma^{(2)})\to C$ is an isomorphism of $A$-graphs that preserves the colorings: $\chi_{1}=\chi_{2}\circ \psi$. For technical reasons that will be clear later, we call a coloring of a rooted $A$-graph $(\Gamma, \gamma_{0})$ any map $\chi:V(\Gamma)\setminus\{\gamma_{0}\}\to C$, and we use the notation $(\Gamma, \gamma_{0}; \chi)$ to stress the coloring function. For such colored rooted $A$-graph $(\Gamma, \gamma_{0}; \chi)$, we may regard an end-cone $\Gamma(v,\gamma_{0})$ as a colored $A$-graph with coloring map $\chi\mid_{V(\Gamma(v,\gamma_{0}))}$ given by the restriction of $\chi$ to the set of vertices of the end-cone $\Gamma(v,\gamma_{0})$. We therefore say that the colored $A$-graph $(\Gamma, \gamma_{0}; \chi)$ is \emph{context-free} if there are finitely many end-cones up to end-isomorphism (in the aforementioned sense). We stick to the notation introduced in Definition~\ref{def: free product}, and let us suppose that $\Theta_{1}$ ($\Theta_{2}$) is a collection of $A_{1}$-graphs ($A_{2}$-graphs) with $A_{1}\cap A_{2}=\emptyset$ and with gluing maps $\Psi$. We say that $\Gamma^{(j)}\in\Theta_{1}$ ($\Lambda^{(j)}\in\Theta_{2}$) is context-free if the colored rooted $A_{1}$-graph $(\Gamma^{(j)}, \gamma_{0}^{j}, \psi^{1}_{j})$ ($A_{2}$-graph $(\Lambda^{(j)}, \lambda_{0}^{j}, \psi^{2}_{j})$) is context-free. If all the colored $A_{i}$-graphs of $\Theta_{1}\cup \Theta_{2}$ are context-free, we say that $(\Theta_{1},\Theta_{2},\Psi)$ is a \emph{context-free ensemble}. We have the following theorem.
\begin{theorem}\label{theo: free product is context-free}
For a context-free ensemble $(\Theta_{1},\Theta_{2},\Psi)$, the free-product $\Theta_{1}\star_{\Psi}\Theta_{2}$ is a context-free $A_{1}\cup A_{2}$-graph. 
\end{theorem}
\begin{proof}
By Lemma~\ref{lem: end-cones of free product}, using the same notation, we have that an end-cone $\Omega(v^{*}\gamma_{i}^{j}\lambda_{t}^{s},\alpha)$ is obtained from the disjoint union 
\[
\Lambda^{s}(\lambda^{s}_{t}, \lambda^{s}_{0})\cup\biguplus_{\substack{z\in V(\Lambda^{s}(\lambda^{s}_{t}, \lambda^{s}_{0}))\\ \gamma^{h}_{0}=\psi^{2}_{s}(z)}} \Gamma^{(h)}(\uparrow_{\Psi})
\]
an then by identifying $z$ with $\psi^{2}_{s}(z)$. We prove that we have finitely many end-cones up to end-isomorphism (in the usual sense). Now, let $\Omega(v^{*}\gamma_{i}^{j}\lambda_{\ell}^{s},\alpha)$ be another end-cone for which there is an end-isomorphism $\varphi: \Lambda^{s}(\lambda^{s}_{t}, \lambda^{s}_{0})\to \Lambda^{s}(\lambda^{s}_{\ell}, \lambda^{s}_{0})$ of end-cones of the same $A_{2}$-graph $\Lambda^{(s)}$ that preserves the coloring. Since the coloring map $\psi^{2}_{s}$ is commuting with $\varphi$, for any $z\in V(\Lambda^{s}(\lambda^{s}_{t}, \lambda^{s}_{0}))$ we have that $\psi^{2}_{s}(z)=\psi^{2}_{s}(\varphi(z))=\gamma_{0}^{h}$, so both $z$ and $\varphi(z)$ have the same color, hence in both the vertex $z$ and $\varphi(z)$ we are gluing the same graph $\Gamma^{(h)}(\uparrow_{\Psi})$. Therefore, it is straightforward to show that $\varphi$ may be extended 
\[
\widehat{\varphi}:\left(\Lambda^{s}(\lambda^{s}_{t}, \lambda^{s}_{0})\cup\biguplus_{\substack{z\in V(\Lambda^{s}(\lambda^{s}_{t}, \lambda^{s}_{0}))\\ \gamma^{h}_{0}=\psi^{2}_{s}(z)}} \Gamma^{(h)}(\uparrow_{\Psi})\right)\to\left( \Lambda^{s}(\lambda^{s}_{\ell}, \lambda^{s}_{0})\cup\biguplus_{\substack{z\in V(\Lambda^{s}(\lambda^{s}_{\ell}, \lambda^{s}_{0}))\\ \gamma^{h}_{0}=\psi^{2}_{s}(z)}} \Gamma^{(h)}(\uparrow_{\Psi})\right)
\]
by locally extending $\varphi$ from each disjoint copy $\Gamma^{(h)}(\uparrow_{\Psi})$ associated to $z$, to the same graph $\Gamma^{(h)}(\uparrow_{\Psi})$ associated to $\varphi(z)$. Now, since we are performing the identification $z=\psi^{2}_{s}(z)$ in the domain of $\widehat{\varphi}$ and this is the same as identifying $\varphi(z)$ with $\psi^{2}_{s}(\varphi(z))$, we have that $\widehat{\varphi}$ induces an isomorphism $\phi: \Omega(v^{*}\gamma_{i}^{j}\lambda_{t}^{s},\alpha)\to \Omega(v^{*}\gamma_{i}^{j}\lambda_{\ell}^{s},\alpha)$. Note that the frontier vertices of $\Omega(v^{*}\gamma_{i}^{j}\lambda_{t}^{s},\alpha)$ are (up to renaming) the same as $\Lambda^{s}(\lambda^{s}_{t}, \lambda^{s}_{0}))$, and similarly for $\Omega(v^{*}\gamma_{i}^{j}\lambda_{\ell}^{s},\alpha)$. Hence, $\phi$ is mapping the frontier vertices of $\Omega(v^{*}\gamma_{i}^{j}\lambda_{t}^{s},\alpha)$ into the frontier vertices of $\Omega(v^{*}\gamma_{i}^{j}\lambda_{\ell}^{s},\alpha)$, hence they are end-isomorphic. We have just shown that the end-isomorphism classes of the end-cones of $\Omega$ of vertices colored by $2$ depends on the end-isomorphism (preserving the coloring) classes of the graphs $\{\Lambda^{(1)}, \ldots, \Lambda^{(m)}\}$ of $\Theta_{2}$, which are finitely many. Similarly, one may show the same for the end-cones of the vertices colored by $1$, and this concludes the proof of the theorem.
\end{proof}

\begin{figure}[ht]
\includegraphics[scale=.7]{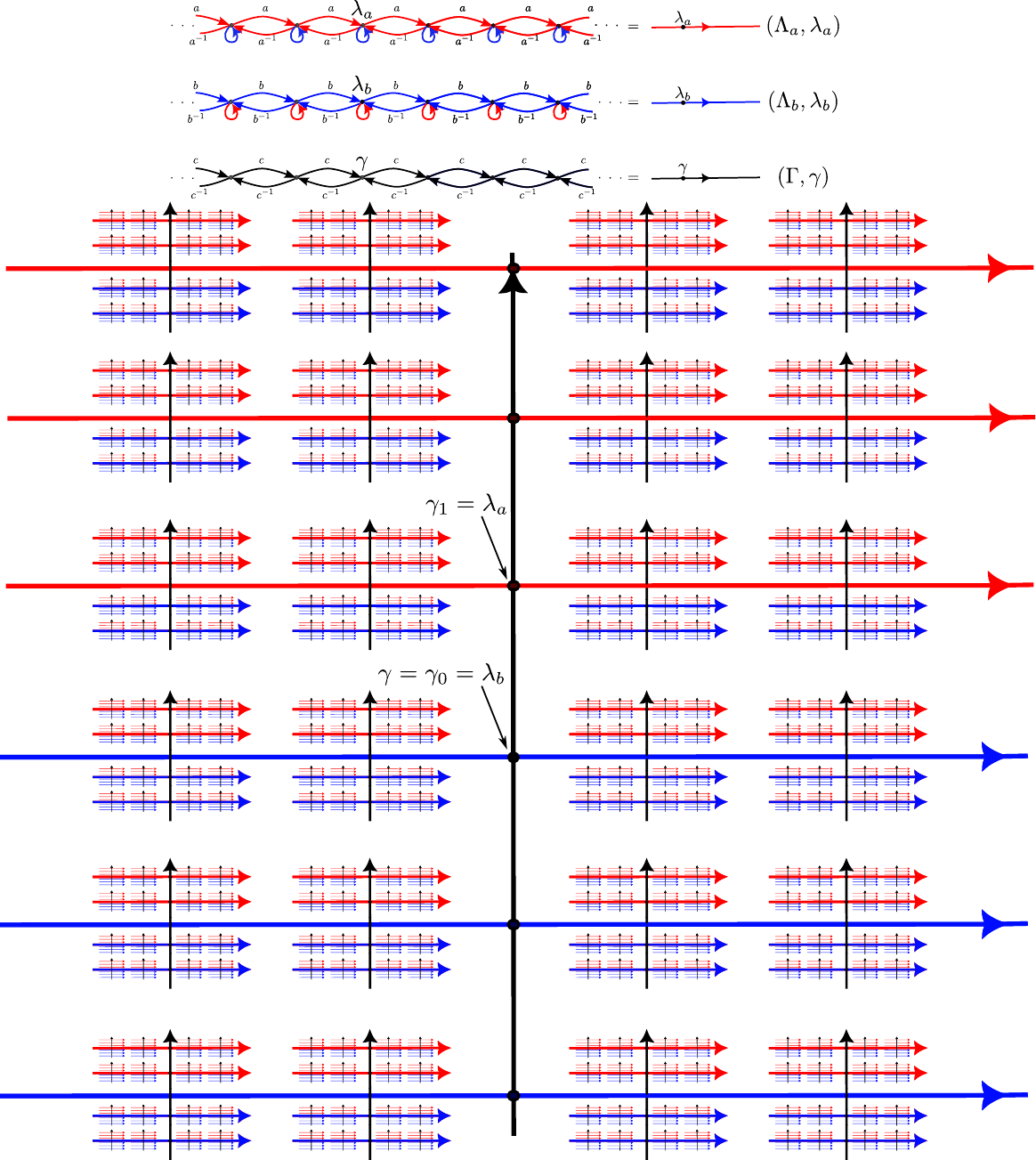}
\caption{On the upper part the graphs $\Theta_{1}, \Theta_{2}$, in the lower picture the antenna graph $(\Theta_{1}\star_{\Psi}\Theta_{2}, \gamma)$.}  \label{fig: antenna}
\end{figure}

\section{Some examples using the free product of graphs}\label{sec: non-residually}
In this section we use the free product of context-free graphs to construct some examples of interesting {\bf CF-TR} groups. In the first example, the antenna graph, we show that in general transition groups of context-free graphs are not residually-finite. In the second one, the comb graph, we prove that the associated transition group is $\mathbb{Z}\wr\mathbb{Z}$, so a non-finitely generated group. Hence, it is a solvable co-context-free group that is not poly-context-free, showing that these classes of groups are different and solving the first point of Question~\ref{prob: z infinity and unbound finite groups}. The third example addresses the second point of Question~\ref{prob: z infinity and unbound finite groups}: it provides an infinite torsion subgroup, so that the group have arbitrarily large finite subgroups.
\begin{defn}[The Antenna graph]
Let $A_{1}=\{c,c^{-1}\}$, $A_{2}=\{a,a^{-1},b,b^{-1}\}$. Consider the rooted $A_{1}$-graph $(\Gamma, \gamma)$ which is the line 
\[
\gamma_i\mapright{c}\gamma_{i+1},\; \gamma_{i+1}\mapright{c^{-1}}\gamma_i,\quad \forall i\in\mathbb{Z}, \mbox{ where the root }\gamma=\gamma_{0}
\]
and the two rooted $A_{2}$-graphs $(\Lambda_{a},\lambda_{a}), (\Lambda_{b},\lambda_{b})$ which are lines with loops defined by 
\[
\lambda_{a,i}\mapright{b}\lambda_{a,i},\; \lambda_{a,i}\mapright{b^{-1}}\lambda_{a,i},\,\lambda_{a,i}\mapright{a}\lambda_{a,i+1},\; \lambda_{a,i+1}\mapright{a^{-1}}\lambda_{a,i},  \quad \forall i\in\mathbb{Z},\,\,\lambda_{a,0}=\lambda_{a}
\]
\[
\lambda_{b,i}\mapright{a}\lambda_{b,i},\; \lambda_{b,i}\mapright{a^{-1}}\lambda_{b,i},\,\lambda_{b,i}\mapright{b}\lambda_{b,i+1},\; \lambda_{b,i+1}\mapright{b^{-1}}\lambda_{b,i},  \quad \forall i\in\mathbb{Z},\,\,\lambda_{b,0}=\lambda_{b}
\]
see the upper part of Fig.~\ref{fig: antenna}. The gluing map $\Psi$ is given by the functions defined as follows:
\[
\psi^{1}:V(\Gamma)\setminus\{\gamma\}\to \{\lambda_{a}, \lambda_{b}\},\,\,\psi^{1}(\gamma_{i})=\begin{cases}
\lambda_{b}&\mbox{ if }i\le 0  \\
\lambda_{a}& \mbox{ otherwise}
\end{cases}
\]
and
\[
\psi^{2}_{a}:V(\Lambda_{a})\setminus\{\lambda_{a}\}\to \{\gamma\}, \,\, \psi^{2}_{b}:V(\Lambda_{b})\setminus\{\lambda_{b}\}\to \{\gamma\}
\]
that are constant maps. By putting $\Theta_{1}=\{(\Gamma, \gamma)\}$ and $\Theta_{2}=\{(\Lambda_{a},\lambda_{a}), (\Lambda_{b},\lambda_{b})\}$, the rooted $A_{1}\cup A_{2}$-graph $(\mathfrak{A}, \gamma)=(\Theta_{1}\star_{\Psi}\Theta_{2}, \gamma)$ is called the antenna graph, see Fig.~\ref{fig: antenna}. Moreover, it is easy to see that $(\Theta_{1}, \Theta_{2}, \Psi)$ is a context-free-ensemble, thus by Theorem~\ref{theo: free product is context-free} $(\mathfrak{A}, \gamma)$ is a context-free graph. Note that with the convention of the free product that we have given, the root $\gamma$ of $\Theta_{1}\star_{\Psi}\Theta_{2}$ is identified with $\lambda_{b}$.
\end{defn}
By the tree structure, the set of vertices can be identified using words in $A_{1}\cup A_{2}$ by associating to a vertex the unique geodesic connecting $\gamma$ to such vertex. It is easy to see that the action of the alphabet $A_{1}\cup A_{2}$ on the set of vertices is defined as follows. Suppose we are in the vertex $v=wx^{s}$, $s\neq 0$, in case $x\in A_{1}\cup A_{2}$ we have
\begin{equation}\label{eq: action}
 \begin{aligned}
&wc^{s} \cdot b^{t}=\begin{cases}
wc^{s}, & \mbox{ if }s>0\\
wc^{s}b^{t}, & \mbox{ otherwise }
\end{cases},\,\, 
wc^{s} \cdot a^{t}=\begin{cases}
wc^{s} a^{t}, & \mbox{ if }s\ge 0\\
wc^{s}, & \mbox{ otherwise }
\end{cases}\,\,\\
& wa^{s} \cdot b^{t}=wa^{s}, \,\,
wb^{s} \cdot a^{t}=wb^{s}, \,\, wa^{s} \cdot c^{t}=wa^{s}c^{t}, \,\,
wb^{s} \cdot c^{t}=wb^{s}c^{t}, \,\,\gamma \cdot a^{s}=\gamma, \gamma\cdot b^{s}=\gamma b^{s}
\end{aligned}
\end{equation}
We now prove that $\mathcal{G}(\mathfrak{A})$ is not residually finite, but first we need the following lemma. 
\begin{lemma}\label{lem: relations antenna}
For each $k,s\in\mathbb{Z}$, $t,p\ge 0$ we have that 
$
\left[ c^{-t}a^{s}c^{t}, c^{p}b^{k}c^{-p}\right]=\mathds{1}_{\mathcal{G}(\mathfrak{A})}.
$
\end{lemma}
\begin{proof}
Note that the statement is equivalent to $g=\left[ a^{s}, c^{\alpha}b^{k}c^{-\alpha}\right]=a^{s}c^{\alpha}b^{k}c^{-\alpha}a^{-s}c^{\alpha}b^{-k}c^{-\alpha}=\mathds{1}_{\mathcal{G}(\mathfrak{A})}$ for every $\alpha\ge 0$ and $k,s\neq 0$. 
We may also assume that $\alpha>0$ since the relation $\left[ a,b\right]=\mathds{1}_{\mathcal{G}(\mathfrak{A})}$ can be easily verified.\

Note that $\gamma\cdot g=\gamma$. We show that for any vertex $v=wx^{m}$, $m\neq 0$ $x\in A_{1}\cup A_{2} $ we have $v\cdot g=v$ from which the statement follows. We have the following cases:
\begin{itemize}
\item $m>0$, $x=c$, using (\ref{eq: action}) we have 
\begin{align*}
wc^{m}\cdot (a^{s}c^{\alpha}b^{k}c^{-\alpha}a^{-s}c^{\alpha}b^{-k}c^{-\alpha})&=wc^m a^s c^{\alpha}\cdot (b^{k}c^{-\alpha}a^{-s}c^{\alpha}b^{-k}c^{-\alpha})=wc^m a^s c^{\alpha}\cdot (c^{-\alpha} a^{-s}c^{\alpha}b^{-k}c^{-\alpha})\\
&=wc^{m}\cdot (c^{\alpha}b^{-k}c^{-\alpha})=wc^{m},
\end{align*}
\item $m<0$, $x=c$, using (\ref{eq: action}) we have $wc^{m}\cdot g=wc^{{m+\alpha}}\cdot (b^{k}c^{-\alpha}a^{-s}c^{\alpha}b^{-k}c^{-\alpha})$. Now we have the following cases:
\begin{itemize}
\item if $m+\alpha>0$ we have $wc^{m+\alpha}\cdot (b^{k}c^{-\alpha}a^{-s}c^{\alpha}b^{-k}c^{-\alpha})=wc^{m}\cdot (c^{\alpha}b^{-k}c^{-\alpha})=wc^{m}$.
\item if $m+\alpha<0$ we have $wc^{m+\alpha}\cdot (b^{k}c^{-\alpha}a^{-s}c^{\alpha}b^{-k}c^{-\alpha})=wc^{m+\alpha} b^{k}\cdot (c^{-\alpha}a^{-s}c^{\alpha}b^{-k}c^{-\alpha})=wc^{m+\alpha}b^{k}c^{-\alpha}\cdot (a^{-s}c^{\alpha}b^{-k}c^{-\alpha})=wc^{m+\alpha}b^{k}c^{-\alpha}\cdot (c^{\alpha}b^{-k}c^{-\alpha})=wc^{m}$.
\item if $m+\alpha=0$, by (\ref{eq: action}) we have that in the most general form, the vertex $v$ is either of the form $v=v'c^{-\ell}b^{t}c^{-\alpha}$ or $v=v'c^{\ell}a^{t}c^{-\alpha}$ for some $\ell>0$. Let us consider these two further cases:
\begin{itemize}
\item In the first case we have:
 $$v'c^{-\ell}b^{t}c^{-\alpha}\cdot (a^{s}c^{\alpha}b^{k}c^{-\alpha}a^{-s}c^{\alpha}b^{-k}c^{-\alpha})=v'c^{-\ell}b^{t}\cdot (b^{k}c^{-\alpha}a^{-s}c^{\alpha}b^{-k}c^{-\alpha})$$
Now, if $t\neq -k$, then $v'c^{-\ell}b^{t}\cdot (b^{k}c^{-\alpha}a^{-s}c^{\alpha}b^{-k}c^{-\alpha})=v'c^{-\ell}b^{t+k}c^{-\alpha}\cdot (a^{-s}c^{\alpha}b^{-k}c^{-\alpha})=v'c^{-\ell}b^{t}c^{-\alpha}=v$. Otherwise, suppose that $t=-k$. In this case we have 
\begin{align*}
& v'c^{-\ell}b^{t}\cdot (b^{k}c^{-\alpha}a^{-s}c^{\alpha}b^{-k}c^{-\alpha})=v'c^{-\ell-\alpha}\cdot (a^{-s}c^{\alpha}b^{-k}c^{-\alpha})=\\
&=v'c^{-\ell}\cdot (b^{-k}c^{-\alpha})=v'c^{-\ell}b^{-k}c^{-\alpha}=v.
\end{align*}
Hence, $v\cdot g=v$.
\item In the second case we the following computation:
\begin{align*}
&v'c^{\ell}a^{t}c^{-\alpha}\cdot (a^{s}c^{\alpha}b^{k}c^{-\alpha}a^{-s}c^{\alpha}b^{-k}c^{-\alpha})=\\
&= v'c^{\ell}a^{t}\cdot (b^{k}c^{-\alpha}a^{-s}c^{\alpha}b^{-k}c^{-\alpha})=v'c^{\ell}a^{t}\cdot (c^{-\alpha}a^{-s}c^{\alpha}b^{-k}c^{-\alpha})=\\
&= v'c^{\ell}a^{t}c^{-\alpha}\cdot (a^{-s}c^{\alpha}b^{-k}c^{-\alpha})=v'c^{\ell}a^{t}\cdot (b^{-k}c^{-\alpha})=v'c^{\ell}a^{t}c^{-\alpha}=v
\end{align*}
\end{itemize}
\end{itemize}
\item if $x=b$ we have the following computation:
\begin{align*}
&wb^{m}\cdot (a^{s}c^{\alpha}b^{k}c^{-\alpha}a^{-s}c^{\alpha}b^{-k}c^{-\alpha})=wb^{m}c^{\alpha}\cdot (b^{k}c^{-\alpha}a^{-s}c^{\alpha}b^{-k}c^{-\alpha})=\\
&= wb^{m}\cdot (a^{-s}c^{\alpha}b^{-k}c^{-\alpha})= wb^{m}c^{\alpha}\cdot (b^{-k}c^{-\alpha})=wb^{m}
\end{align*}
\item the last case is when $x=a$. In this case we have to subdivide the following two subcases:
\begin{itemize}
\item if $m+s\neq 0$, for which we have:
\begin{align*}
&wa^{m}\cdot (a^{s}c^{\alpha}b^{k}c^{-\alpha}a^{-s}c^{\alpha}b^{-k}c^{-\alpha})=wa^{m+s}c^{\alpha}\cdot (b^{k}c^{-\alpha}a^{-s}c^{\alpha}b^{-k}c^{-\alpha})=\\
&=wa^{m+s}\cdot (a^{-s}c^{\alpha}b^{-k}c^{-\alpha})=wa^{m}\cdot (c^{\alpha}b^{-k}c^{-\alpha})=wa^{m}.
\end{align*}
\item if $m+s= 0$, then in its most general form the vertex $v$ is of the form$v=v'c^{t}a^{-s}$ with $t\ge 0$. In case $t>0$ we have:
\begin{align*}
&v'c^{t}a^{-s}\cdot (a^{s}c^{\alpha}b^{k}c^{-\alpha}a^{-s}c^{\alpha}b^{-k}c^{-\alpha})=v'c^{t+\alpha}\cdot (b^{k}c^{-\alpha}a^{-s}c^{\alpha}b^{-k}c^{-\alpha})=\\
&=v'c^{t}\cdot (a^{-s}c^{\alpha}b^{-k}c^{-\alpha})=v'c^{t}a^{-s}
\end{align*}
In case $t=0$ we necessarily have $v'=\gamma$ and also in this case we have $\gamma a^{-s}\cdot (a^{s}c^{\alpha}b^{k}c^{-\alpha}a^{-s}c^{\alpha}b^{-k}c^{-\alpha})=\gamma a^{-s}$.
\end{itemize}
\end{itemize}
\end{proof}
\begin{theorem}
The group $G=\mathcal{G}(\mathfrak{A})$ defined by the antenna graph $\mathfrak{A}$ is not residually finite. 
\end{theorem}
\begin{proof}
Let us assume that $G$ is residually finite. Consider the element
\[
h=\left[ a, c^{-1}bc\right]=ac^{-1}bca^{-1}c^{-1}b^{-1}c
\] 
Note that $h\neq \mathds{1}_{\mathcal{G}(\mathfrak{A})}$ since we have:
\[
\gamma\cdot ac^{-1}bca^{-1}c^{-1}b^{-1}c=\gamma\cdot (c^{-1}bca^{-1}c^{-1}b^{-1}c)=\gamma c^{-1}bca^{-1}c^{-1}b^{-1}c \neq \gamma
\]
Now, since $G$ is residually finite, there is a finite group $H$ and an homomorphism $\varphi: G\to H$ such that $\varphi(h)\neq\mathds{1}_{H} $. Since $H$ is finite there is some positive $\alpha>0$ such that $\varphi(c)^{\alpha}=\varphi(c)^{-1}$. By Lemma~\ref{lem: relations antenna} we have $\left[ a, c^{\alpha}bc^{-\alpha}\right]=\mathds{1}_{\mathcal{G}(\mathfrak{A})}$, hence:
\begin{align*}
\mathds{1}_{H}&=\varphi\left(\left[ a, c^{\alpha}bc^{-\alpha}\right] \right)=\left[ \varphi(a), \varphi(c)^{\alpha}\varphi(b)\varphi(c)^{-\alpha}\right]=\left[ \varphi(a), \varphi(c)^{-1}\varphi(b)\varphi(c)\right]=\\
&= \varphi\left(\left[ a, c^{-1}bc\right] \right)=\varphi(h)\neq \mathds{1}_{H}
\end{align*}
a contradiction, hence $G$ is not residually finite. 
\end{proof}
In light of Theorem~\ref{theo: virtually subgroups} and the preceding theorem, we observe that, if Brough Conjecture~\ref{cj: Brough1} regarding poly-context-free groups is true, then the class of all poly-context-free groups would be strictly contained within the class {\bf CF-TR}. Although we are not able to show that the class of poly-context-free is included into our class, in the next example we show that these two classes are different. This is done by defining the following context-free graph. 
\begin{figure}[ht]
\includegraphics[scale=.6]{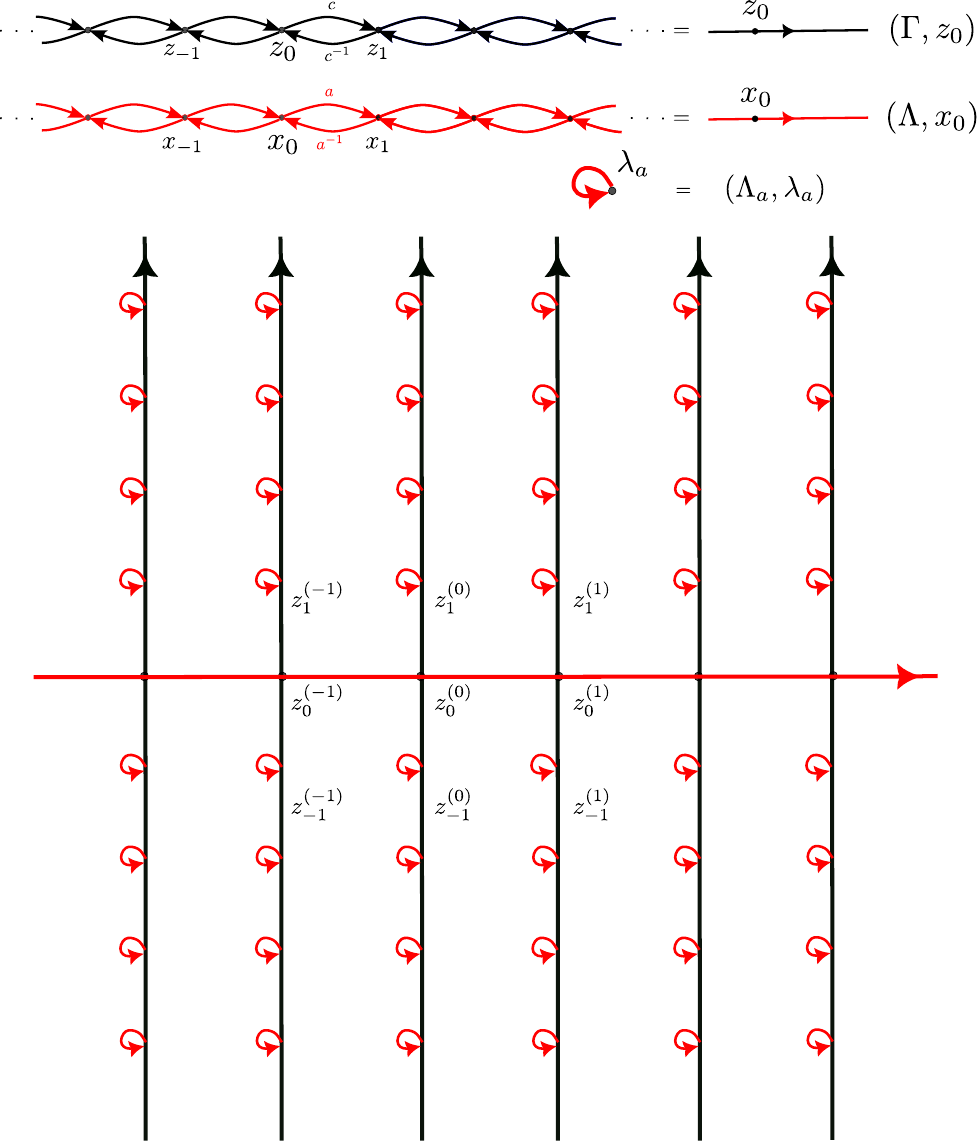}
\caption{On the upper part the graphs $\Theta_{1}, \Theta_{2}$, in the lower picture the comb graph $(\Theta_{1}\star_{\Psi}\Theta_{2}, \gamma)$.}  \label{fig: comb}
\end{figure}
\begin{defn}[The comb graph] \label{def:comb}
Let $A_{1}=\{c,c^{-1}\}$, $A_{2}=\{a,a^{-1}\}$. Let us define the rooted $A_{1}$-graphs $(\Gamma, z_{0})$ whose underlying graph is the line 
\[
z_i\mapright{c} z_{i+1},\;z_{i+1}\mapright{c^{-1}}z_i,\quad \forall i\in\mathbb{Z},
\]
The other two rooted $A_{2}$-graphs $(\Lambda_{a},\lambda_{a}), (\Lambda, x_{0})$ are defined as follows. The graph $(\Lambda_{a}, \lambda_{a})$ consists of just one vertex $\lambda_{a}$ with loops labeled by $a,a^{-1}$. Instead, $(\Lambda_{a},x_{0})$ is the defined as the line on $A_{2}$:
\[
x_i\mapright{a} x_{i+1},\;x_{i+1}\mapright{a^{-1}}x_i,\quad \forall i\in\mathbb{Z},
\]
See the upper part of Fig.~\ref{fig: comb}. The gluing map $\Psi$ is given by the following functions. The gluing maps are given by the following constant maps 
\[
\psi_{1}:V(\Gamma)\setminus\{z_{0}\}\to \{\lambda_{a}\},\,\quad\psi_{1}(z)=\lambda_{a}
\]
\[
\psi_{2}:V(\Lambda)\setminus\{x_{0}\}\to \{z_{0}\},\,\quad\psi_{2}(x_{i})= z_{0}
\]
By putting $\Theta_{1}=\{(\Gamma, z_{0})\}$ and $\Theta_{2}=\{(\Lambda,x_{0}), (\Lambda_{a}, \lambda_{a})\}$, we see that $(\Theta_{1}, \Theta_{2}, \Psi)$ is a context-free ensemble, thus the free product $(\mathfrak{C}, \gamma)=(\Theta_{1}\star_{\Psi}\Theta_{2}, z_{0})$ defines a context-free graph called the \textbf{comb graph}, see Fig.~\ref{fig: comb}. Note that with the convention of the free product that we have given, the root of this graph is $z_{0}$ which is also identified with $x_{0}$. In $\mathfrak{C}$ there are infinite copies of the $c$-line, we denote by $z_{i}^{(j)}$ the $i$-th vertex of the $j$-th copy of this line. Graphically, two vertices $z_{i}^{(j)}$, $z_{i}^{(k)}$ are at the same ``$i$-th level'', but they belong to the $j$-th and $k$-th $c$-line, respectively. The $a$-line intersects at the vertices $z_{0}^{(j)}$.
\end{defn}
\begin{prop}\label{prop:solvable-z-infinito}
The transition group $G = \mathcal{G}(\mathfrak{C})$ is isomorphic to $\mathbb{Z} \wr \mathbb{Z}$. 
In particular, $G$ is not finitely presented, solvable and contains a copy of the infinite-rank free abelian group $\mathbb{Z}^{\infty}$. Hence it is not poly–context-free.
\end{prop}
\begin{proof}
By Proposition \ref{prop:schreier} we need to find a core-free subgroup of $\mathbb{Z}\wr\mathbb{Z}$ whose corresponding Schreier graph is exactly the graph depicted in Fig.~\ref{fig: comb}. But this figure is exactly Fig.~1 in \cite{boudec-mattebon}.
\end{proof}

\begin{figure}[ht]
\includegraphics[scale=.7]{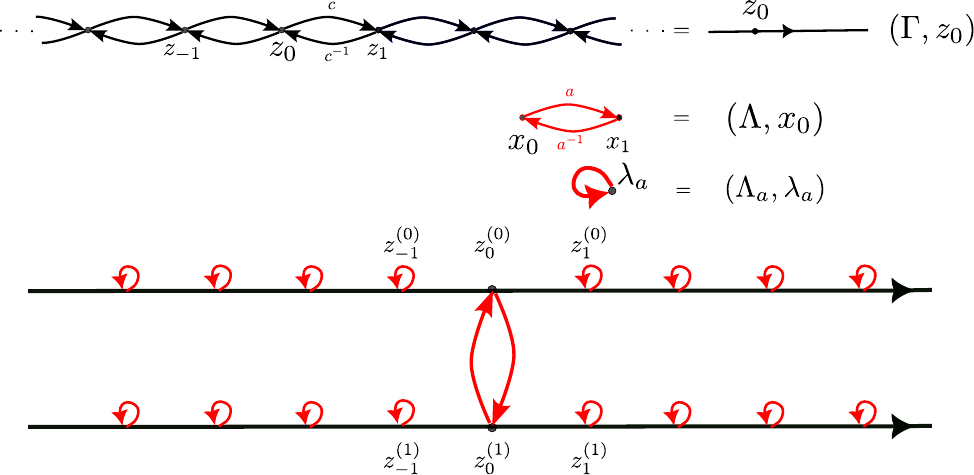}
\caption{The context-free graph $\mathfrak{T}$ obtained by taking the free product of the graphs in the upper part, where the gluing maps are defined by taking $\psi^1(z_i)=\lambda_a$ for all $z_i\neq z_0$, and $\psi^2(x_1)=z_0$.}  \label{fig: torsion}
\end{figure}
We finally address the second point of Question~\ref{prob: z infinity and unbound finite groups} regarding the possibility for a transition group to have arbitrarily large finite subgroups. We prove that such examples may exists by showing a group in {\bf CF-TR} having an infinite torsion subgroup. The associated graph $\mathfrak{T}$ on the alphabet $A=\{a,c,a^{-1}, c^{-1}\}$ is depicted in Fig.~\ref{fig: torsion}. Using the notation in Subsection~\ref{subsec: perturbing} we notice that $\mathfrak{T}$ is locally-quasi-transitive with infinite orbits by taking $n=1$. We have the following proposition. 
\begin{prop} \label{prop:boundary-infinite-torsion}
    The boundary subgroup $\mathrsfs{B}_{1}$ is an infinite torsion normal subgroup of $\mathcal{G}(\mathfrak{T})$ whose elements have order two. Moreover, $ G/_{\mathrsfs{B}_{1}} \simeq \mathbb{Z}$.
\end{prop}
\begin{proof}
    Let $\chi_c: A^{*}\to \mathbb{Z}$ be the homomorphism counting the $c$, i.e., $\chi_c(c)=-\chi_c(c^{-1})=1$ and deleting all the occurrences of $a$.
    Note that $\mathfrak{T}$ is locally-quasi-transitive with $\mathfrak{T}\setminus D_1(x_0)=\Lambda_1\setminus D_1(y_0)$, where $\Lambda_1$ is the transitive graph formed by two $c$-lines with two loops labeled by $a,a^{-1}$ everywhere. By Lemma~\ref{lem: infinite orbits}, the language $\mathcal{O}_1$ is formed by words that label a circuit at every vertex of $\Lambda_1$. It is easy to see that this is exactly the set of words $u\in A^*$ with $\chi_c(u)=0$. Therefore, for any vertex $z_i^{(j)}$, $j\in\{0,1\}$ of $\mathfrak{T}$ and $u\in\mathcal{O}_1$ we have that either $z_i^{(j)}\cdot u=z_i^{(j)}$, or $z_i^{(j)}\cdot u=z_i^{(j+1\mod 2)}$, where the second case occurs if and only the walk $z_i^{(j)}\mapright{u} z_i^{(j+1\mod 2)}$ passes from the edges $z_0^{(j)}\mapright{x} z_0^{(j+1\mod 2)}$, $x\in\{a,a^{-1}\}$, an odd number of times. In particular we have $z_i^{(j)}\cdot u^{2}=z_i^{(j)}$ for every vertex of $\mathfrak{T}$, whence every element of $\mathrsfs{B}_{1}$ has torsion of order two. The boundary subgroup $\mathrsfs{B}_{1}$ is infinite. Indeed, take the elements $h_n=c^{-n} a c^n\in \mathrsfs{B}_{1}$, $n\in\mathbb{Z}$. These are all distinct elements in $\mathcal{G}(\mathfrak{T})$ since for any $j\in\{0,1\}$ we have
    \[
        z_m^{(j)}\cdot h_n=\begin{cases}
            z_m^{(j+1\mod 2)}\,&\mbox{ if }m=n\\
            z_m^{(j)}\,&\mbox{ otherwise }
        \end{cases}
    \]
    The last statement $ G/_{\mathrsfs{B}_{1}} \simeq \mathbb{Z}$ is a consequence of $\mathcal{G}(\Lambda_1)=\langle c\rangle\simeq \mathbb{Z}\le \mathcal{G}(\mathfrak{T})$ and Theorem~\ref{theo: structural locally quasi-trans}. 
\end{proof}
We may explicitly characterize the word problem $\mathcal{L}(\mathfrak{T})$. We first need some definitions. For a word $u\in A^*$, let us denote by $u[i]$ the letter of $u$ occurring at position $i$, and by $u[:i]$ let us denote the prefix of $u$ of length $u$. We say that an occurrence of the letter $a=u[i]$, or $a^{-1}=u[i]$, is at level $m$ if $\chi_c(u[:i])=m$. Since there are just two vertices, namely $z_0^{(1)}, z_0^{(1)}$ in which a walk may enter, it is not difficult to see that for a fixed $m$ the number of $a, a^{-1}$ at level $m$ in $u$ represents the number of times a walk starting from $z_{-m}^{(k)}$, $k\in\{0,1\}$, passes through the ``bridge" given by the edges $z_0^{(j)}\mapright{x} z_0^{(j+1\mod 2)}$, $x\in\{a,a^{-1}\}$. Therefore, a word $u\in A^*$ represents the identity in $\mathcal{G}(\mathfrak{T})$ if and only if $\chi_c(u)=0$ and for each $m\in\mathbb{Z}$ the number of $a,a^{-1}$ at level $m$ is even. Note that we clearly have finitely many $m$ with non-zero number of $a,a^{-1}$. We know that this language is {\bf CoCF}, however, is it also poly-context-free? Let us state this in the following open problem.
\begin{quest}
    Is the set of words $u\in A^*$ with $\chi_c(u)=0$ and having a number of $a,a^{-1}$ that is even at each level a poly-context-free language? This would disprove Brough Conjecture \ref{cj: Brough1} 
\end{quest}


\begin{quest}
Motivated by the previous constructions, it would be interesting to find examples of {\bf CF-TR} groups which are not embeddable in Thompson $V$. If they existed, they would then provide counterexamples to
\hyperref[thm:lehnert]{Lehnert Conjecture}.
\end{quest}


\section*{Acknowledgments}
The authors are members of the Gruppo Nazionale per le Strutture Algebriche, Geometriche e le loro Applicazioni (GNSAGA) of the Istituto Nazionale di Alta Matematica (INdAM).

The second author is also a member of the PRIN 2022 ``Group theory and its applications'' research group of the Ministero dell'universit\`a e della ricerca and gratefully acknowledges its support. 

The third author acknowledges support from the Grant QUALIFICA by Junta de Andalucía grant number QUAL21 005 USE and from the research grant PID2022-138719NA-I00 (Proyectos de Generación de Conocimiento 2022) financed by
the Spanish Ministry of Science and Innovation.

We are deeply grateful to Corentin Bodart for helpful discussions on the first draft of this manuscript, especially for Proposition~\ref{prop:schreier}, Proposition~\ref{prop:solvable-z-infinito} and for Section~\ref{sec:torsion}.

\end{document}